\documentclass[psamsfonts,11pt]{amsart}
\usepackage[left=1in,right=1in,top=1in,bottom=1in]{geometry}
\usepackage{times}
\usepackage[all,cmtip]{xy}
\usepackage{amssymb}
\usepackage[english]{babel}
\usepackage{mathrsfs}

\newtheorem{theorem}{Theorem}[section]
\newtheorem{lemma}[theorem]{Lemma}
\newtheorem{proposition}[theorem]{Proposition}
\newtheorem{corollary}[theorem]{Corollary}

\theoremstyle{definition}
\newtheorem{definition}[theorem]{Definition}
\newtheorem{definitions}[theorem]{Definitions}
\newtheorem{example}[theorem]{Example}

\newtheorem{conjecture}[theorem]{Conjecture}

\newtheorem{remark}[theorem]{Remark}

\newtheorem{question}[theorem]{Question}

\newtheorem{notation}[theorem]{Notation}

\newtheorem{claim}{Claim}

\let\bb\mathbb

\let\frac=\dfrac
\let\cal\mathcal
\def\cl{\mathop{\rm Cl}\nolimits}
\def\inte{\mathop{\rm Int}\nolimits}

\def\Cl{\mathrm{Cl}}
\def\Int{\mathrm{Int}}
\def\B{\mathcal{B}}
\def\A{\mathcal{A}}

\def\Tp{Todd complete}
\def\Op{Oxtoby complete}
\def\SOp{S\'anchez-Okunev complete}
\def\Telgarsky{Telg\'arsky complete}
\def\N{\mathbb{N}}
\def\R{\mathbb{R}}

\def\cc{countably compact}
\def\ccness{countable compactness}
\def\sp#1{strongly {#1} complete}
\def\sOp{\sp{Oxtoby}}
\def\sSOp{\sp{S\'anchez-Okunev}}
\def\sTep{\sp{Telg\'arsky}}
\def\sTp{\sp{Todd}}

\begin{document}
\title
[Completeness and compactness properties]
{Completeness and compactness properties in metric 
spaces, topological groups and function spaces} 
\author[A. Dorantes-Aldama]{Alejandro Dorantes-Aldama}
\address{Department of Mathematics, Faculty of Science, Ehime University,
  Matsuyama 790-8577, Japan}
\email{alejandro\_dorantes@ciencias.unam.mx}
\thanks{The first listed author was supported by CONACyT: Estancias Posdoctorales al Extranjero propuesta No. 263464}

\author[D. Shakhmatov]{Dmitri Shakhmatov}
\address{Division of Mathematics, Physics and Earth Sciences\\
Graduate School of Science and Engineering\\
Ehime University, Matsuyama 790-8577, Japan}
\email{dmitri.shakhmatov@ehime-u.ac.jp}
\thanks{The second listed author was partially supported by the Grant-in-Aid for Scientific Research~(C) No.~26400091 by the Japan Society for the Promotion of Science (JSPS)}

\dedicatory{Dedicated to Mikhail Tkachenko on the occasion of his 60th anniversary} 

\begin{abstract} 
We prove that many completeness properties coincide in metric spaces, precompact groups and dense subgroups of products of separable metric groups. We apply these results to function spaces $C_p(X,G)$ of $G$-valued continuous functions on a space $X$ with the topology of pointwise convergence, for a separable metric group $G$. Not only the results but also the proofs themselves are novel even in the classical case when $G$ is the real line.

A space $X$ is {\em weakly pseudocompact\/} if it is $G_\delta $-dense in 
at least one of its compactifications. A topological group $G$ is {\em precompact\/} if it is topologically isomorphic to 
a subgroup of a compact group. We prove that every weakly pseudocompact 
precompact 
topological group is pseudocompact, thereby answering positively a question of Tkachenko. 
\end{abstract}

\maketitle

{\em All topological spaces are considered to be Tychonoff and all topological groups Hausdorff.\/}

The symbol $\N$ denotes the set of natural numbers, $\R$ denotes the set of real numbers.   
If $f:X\to Y$ is a function and $A$ is a subset of $Y$, then we let $f^{\leftarrow }(A)=\{x\in X:f(x)\in A\}$. 

Let $X$ be a topological space. For
a subset $A$ of $X$, we use $\Cl_X(A)$ and $\Int_X(A)$ to denote the closure and the interior of $A$ in $X$, respectively.
A set of the form $f^{\leftarrow}(\{0\})$ for some real-valued continuous function $f$ on $X$ is called a {\em zero-set\/} of $X$.
Symbol $\beta X$ denotes the Stone-\v{C}ech compactification of
$X$.

For a topological group $G$, we denote by $\varrho G$ the Raikov completion of $G$.

\section{Introduction}

A product of two metric spaces having Baire property need not be Baire.
This result prompted for a search of proper subclasses $\mathscr{C}$ of the class of Baire spaces 
which 
are 
closed under taking arbitrary Cartesian products; one may call such $\mathscr{C}$ a productively Baire class. The first productively Baire class was introduced by J.C. Oxtoby \cite{O} under the name {\em pseudocomplete\/} spaces. A potentially wider class of such spaces was considered by A.R. Todd \cite{T} who used the same term ``pseudocomplete'' to denote the spaces from his class. In his study of topological games, R. Telg\'arsky \cite{Te} considered another class $\mathscr{C}$ of productively Baire spaces
such that Player II has a Markov winning strategy in the Choquet game on every space from the class $\mathscr{C}$.
Yet another productively Baire property was invented  by F. S\'anchez-Texis and O. Okunev \cite{SO} in their recent work on weakly pseudocompact spaces.

The four properties described above are defined in terms of the existence of a certain sequence
$\{{\cal B}_n:n\in\N\}$ of families $\B_n$ of subsets of $X$ such that each $\B_n$ contains arbitrary small sets with non-empty interior in $X$; see Definition 
\ref{four:properties}. 
The main goal of this paper is to study these four productively Baire properties (as well as their variations obtained by taking all $\B_n$ to be the same)
in metric spaces, topological groups and function spaces.
We also introduce four compactness-type properties closely related to the four completeness properties defined above obtained by disposing of the sequence $\{{\cal B}_n:n\in\N\}$ completely and replacing it 
with
a family ${\cal F}$ of open or zero-sets in $X$ 
such that every decreasing sequence in ${\cal F}$ has non-empty intersection; see Definition \ref{def:compact:properties}.

The second source of inspiration for this paper comes from a question related to weakly pseudocompact spaces introduced
by S. Garc\'ia-Ferreira and
A. Garc\'ia-M\'aynez  in
 \cite{GG}.
\begin{definition}
\label{def:weakly:psc}
\begin{itemize}
\item[(i)]
A subspace $X$ of a space $Y$ is {\em $G_\delta $-dense} in $Y$ 
provided that every non-empty $G_\delta $-subset of $Y$ intersects $X$.
\item[(ii)]
A space $X$ is said to be {\em weakly pseudocompact} if it is $G_\delta $-dense in some compact space $Y$. 
\end{itemize}
\end{definition}

A well-known result of E. Hewitt \cite{H} states that a space $X$ is pseudocompact if and only if
$X$ is $G_\delta $-dense in its Stone-\v{C}ech compactification $\beta X$. 
Therefore, pseudocompact spaces are weakly pseudocompact.

In a private conversation with the second author during the 1st Pan Pacific International Conference on Topology and Applications (PPICTA) held on November 25--30, 2015, 
at Minnan Normal University (Zhangzhou, China),
M. Tkachenko proposed the following conjecture:

\begin{conjecture}
\label{conjecture:Tkachenko}
If a subgroup $G$ of a compact group is $G_\delta$-dense in some compact space,
then $G$ is $G_\delta$-dense also in its Stone-\v{C}ech compactification $\beta G$. 
Equivalently, a weakly pseudocompact precompact topological group is pseudocompact. 
\end{conjecture}

The paper is organized as follows. In Sections \ref{section:1} and \ref{Compactness-like properties} we introduce relevant
completeness and compactness properties, respectively, and explain basic relationships between them. 

Our main results are assembled in Section \ref{results}. We prove that majority of completeness and compactness properties studied in this paper coincide for metric spaces (Theorem \ref{coincidence:for:metric:spaces}), and they coincide 
with the classical \v{C}ech completeness for metric topological groups; see
Theorem \ref{metric:Todd:are:Cech-complete}. 
In Theorem \ref{equivalent:conditions:for:precompact} we show that many of completeness and compactness properties become equivalent to pseudocompactness in precompact groups. The same theorem confirms Conjecture \ref{conjecture:Tkachenko} in the positive; see the equivalence of  items (i) and (ii) in Theorem \ref{equivalent:conditions:for:precompact}.
Finally, we prove in Theorem \ref{equivalence:for:dense:subgroups:without:cech}
that most of completeness and compactness properties coincide for dense subgroups of a Cartesian product of separable metric groups, and we offer complete characterization of when this happens.

Section \ref{C_p-section} is devoted to applications of our main results to function spaces $C_p(X,G)$ of all $G$-valued continuous maps on a space $X$ in the topology of pointwise convergence, for a separable metric topological group $G$.
Highlights here include the coincidence of various completeness properties in $C_p(X,G)$ for an arbitrary separable metric group $G$ and characterization of these completeness properties in terms of internal properties of the space $X$ and that of the topological group $G$, under a mild assumption that $C_p(X,G)$ is dense in $G^X$; see Theorem \ref{Cp:theorem}. In the special case when $G$ is compact metric, these completeness properties become equivalent to pseudocompactness of $C_p(X,G)$; see Corollary \ref{Cp:corolary}.
This strengthens significantly \cite[Theorem 7.2]{DRT}.
In the special case when $G$ is the real line $\mathbb{R}$, our results imply the classical characterization of (Oxtoby) pseudocompleteness of $C_p(X)$ in terms of the properties of $X$ due to V.V. Tkachuk \cite[Theorem 4.1]{Tk}; see Corollary \ref{Theorem:4.1:Tk}. A special feature of our results for function spaces $C_p(X,G)$ is that their proofs are completely devoid of any ``function spaces machinery'' present in the classical proofs. Indeed, only the denseness of $C_p(X,G)$ in $G^X$ is used to derive our $C_p$-results from the correspondent results in Section \ref{results} about dense subgroups in products of separable metric groups.
Therefore, it is fair to say that we offer completely different ``function spaces free'' proofs of results in $C_p$-theory even in the classical case when $G$ is taken to be the real line $\mathbb{R}$.

Section \ref{section:metric} is devoted to proofs of metric Theorems \ref{coincidence:for:metric:spaces} and 
\ref{metric:Todd:are:Cech-complete}. Section \ref{technical:lemmas} contains 
some technical lemmas needed later.
The heart of the paper is in Section \ref{factorization:section} which provides a factorization theorem for (Todd) pseudocompleteness which says that, under certain assumptions, a space having a certain completeness property has ``sufficiently many'' continuous maps onto (Todd) pseudocomplete separable metric spaces.
Carrying  this over to topological groups in Section \ref{sec:polish;factorizable}, we introduce the notion of a Polish factorizable group as such a group 
that has ``sufficiently many'' continuous homomorphisms onto \v{C}ech-complete 
separable metric (aka Polish) groups; see
Definition \ref{def:polish:factorizable}.
Corollary \ref{polish_fact} is the main result in Section \ref{sec:polish;factorizable}. In Section \ref{sec:proofs} we investigate completeness and compactness properties in Cartesian products. A proof of Theorem \ref{equivalence:for:dense:subgroups:without:cech} is given there. Open questions are collected in the last Section \ref{sec:questions}.

\section{(Pseudo)completeness properties}
\label{section:1}

\begin{definition}
\label{def:nested}
We shall say that a sequence $\{B_n:n\in\N\}$ of subsets of a topological space $X$ is {\em nested\/} provided that 
$\Cl_X (B_{n+1}) \subseteq \Int _X (B_n)$ 
for all $n\in\mathbb{N}$.
\end{definition}

\begin{remark}
\label{decreasing:vs:nested}
\begin{itemize}
\item[(i)] Every nested sequence is decreasing.
\item[(ii)] A decreasing sequence consisting of clopen sets is nested.
\end{itemize}
\end{remark}

\begin{remark}
\label{rem:nested}
If $\{B_n:n\in\N\}$ is a nested sequence in a topological space $X$, then
$\bigcap_{n\in\N} B_n=
\bigcap_{n\in\N} \Int_X(B_n)=
\bigcap_{n\in\N} \Cl_X(B_n)$.
\end{remark}

\begin{definition}
\label{def:pi-pseudobase}
Recall that  a family $\cal B$ of subsets of a topological space $X$ 
is called
a {\em pseudobase} of $X$ if it satisfies two conditions:
\begin{itemize}
\item[(i)] 
$\Int_X(B)\not=\emptyset$ for all $B\in \cal B$;
\item[(ii)]
every non-empty open subset of $X$ contains some element of $\cal B$.
\end{itemize}
\end{definition}

\begin{definitions}
\label{def:nested:and:complete:sequences}
Let $X$ be a topological space. For each $n\in\N$ let $\cal{B}_n$ be a family of subsets of $X$.
We shall say that 
a sequence $\{\cal{B}_n: n\in\N\}$ is:
\begin{itemize}
\item[(i)]
 {\em centered\/} provided that 
$\bigcap_{n\in\N}B_n \not= \emptyset$ holds
for every nested sequence $\{B_n:n\in\N\}$ such that 
$B_n \in \cal{B}_n$ for all $n\in\mathbb{N}$;
\item[(ii)]
 {\em complete\/} provided that 
it is centered and each 
$\cal{B}_n$ is a pseudobase for $X$.
\end{itemize}
\end{definitions}

\begin{remark}
\label{sufamily:of:a:centered:family:is:centered}
Let $\{{\cal A}_n:n\in\N\}$ and $\{{\cal B}_n:n\in\N\}$ be two sequences 
of families of subsets of a topological space $X$ such that ${\cal A}_n\subseteq {\cal B}_n$ for each $n\in\N$.

(i)
If $\{{\cal B}_n:n\in\N\}$ is a centered sequence in $X$,
then the sequence $\{{\cal A}_n:n\in\N\}$ is centered as well.

(ii) If $\{{\cal B}_n:n\in\N\}$ is a complete sequence in $X$
and each $\{{\cal A}_n:n\in\N\}$ is a pseudobase for $X$, then 
$\{{\cal A}_n:n\in\N\}$ is a complete sequence in $X$.
\end{remark}

\begin{definition}
\label{four:properties}
A topological space $X$ will be called:
\begin{itemize}
\item[(i)] {\em \Tp\/} if $X$ has a complete sequence;
\item[(ii)] {\em \Op\/} if $X$ has a complete sequence $\{\cal{B}_n: n\in\N\}$
such that each 
$\cal{B}_n$ consists of open subsets of $X$;
\item[(iii)] {\em \SOp\/} if $X$ has a complete sequence $\{\cal{B}_n: n\in\N\}$
such that each 
$\cal{B}_n$ consists of zero-sets in $X$;
\item[(iv)] {\em \Telgarsky\/} if $X$ has a complete sequence $\{\cal{B}_n: n\in\N\}$
such that each 
$\cal{B}_n$ is a base of $X$.
\end{itemize}
\end{definition}

Clearly, $\mbox{\SOp} \to \mbox{\Tp}$ and 
$$
\mbox{\Telgarsky} \to \mbox{\Op} \to \mbox{\Tp}.
$$
The question whether the last implication can be reversed remains open; 
see \cite[Question 6.3]{T}.
Finally, all \Tp\  spaces have the Baire property; see \cite[Theorem 1.3]{T}.

One can also introduce a (theoretically) stronger version of each of the four 
pseudocompleteness properties from Definition \ref{four:properties} by requiring that all
${\cal B}_n$ in the correspondent complete sequence $\{{\cal B}_n:n\in\N\}$ 
are the same.
We shall distinguish these properties by adding the adjective ``strongly'' in front of their names.

The notion of \Op ness  has appeared first  in \cite{O} under the name ``pseudocompleteness'',
while that of \Tp ness has appeared later in \cite{T} (under the same name ``pseudocompleteness'').
The notion from Definition \ref{four:properties}(iii) has appeared 
recently in 
\cite{SO} where it was related to the notion of weak pseudocompactness to be discussed later in 
Section \ref{Compactness-like properties}.

The notion from Definition \ref{four:properties}(iv) was considered first by Telg\'arsky in \cite{Te} who proved in \cite[Theorem 2.9]{Te}
that 
Player II has a Markov winning strategy in the Choquet game on every \Telgarsky\ space (in our terminology). Telg\'arsky also considered a stronger version of his property by requiring that all $\cal B_n$ in the correspondent complete sequence
$\{\cal B_n:n\in\N\}$ are the same; in other words, he considered the class of what we call here strongly \Telgarsky\ spaces.

\begin{lemma}
\label{centered:families:in:G-delta-sets}
Let $\{V_n:n\in\N\}$ be a sequence of subsets of a countably compact space $Y$
and let $X=\bigcap \{V_n:n\in\N\}$.
For every $n\in\N$, let 
\begin{equation}
\label{B:in:V_n}
\B_n=\{B\subseteq X: B\not=\emptyset\mbox{ and }\Cl_Y (B)\subseteq V_n\}.
\end{equation}
Then:
\begin{itemize}
\item[(i)] $\{\B_n:n\in\N\}$ is a centered sequence in $X$;
\item[(ii)] if $\A_n$ is a pseudobase of $X$ such that $\A_n\subseteq \B_n$ for every $n\in\N$, then $\{\A_n:n\in\N\}$ is a complete sequence in $X$.
\end{itemize}
\end{lemma}
\begin{proof}
(i)
Let $\{B_n:n\in\N\}$ be a nested sequence in $X$ such that 
$B_n \in \cal{B}_n$ for all $n\in\mathbb{N}$.
By 
Remark \ref{decreasing:vs:nested}(i),
the sequence $\{B_n:n\in\N\}$ is decreasing,
so
$\{\Cl_Y(B_n):n\in\N\}$ is a decreasing sequence of non-empty closed subsets of $Y$. Since $Y$ is countably compact, 
$\bigcap\{\Cl_Y(B_n):n\in\N\}\not=\emptyset$.
Since $B_n\in\B_n$, it follows from \eqref{B:in:V_n} that
$\Cl_Y(B_n)\subseteq V_n$, so
$\bigcap\{\Cl_Y(B_n):n\in\N\}\subseteq \bigcap \{V_n:n\in\N\}=X$.
Therefore,
\begin{equation}
\label{eq:3:int}
\emptyset\not=
\bigcap_{n\in\N}\Cl_Y(B_n)
=\left(\bigcap_{n\in\N}\Cl_Y(B_n)\right)\cap X
=\bigcap_{n\in\N}\Cl_Y(B_n)\cap X
=
\bigcap_{n\in\N}\Cl_X(B_n).
\end{equation}
Since $\{B_n:n\in\N\}$ is a nested sequence in $X$, it follows from Remark \ref{rem:nested}
that 
$\bigcap\{\Cl_X(B_n):n\in\N\}=\bigcap\{B_n:n\in\N\}$,
so
$\bigcap\{B_n:n\in\N\}\not=\emptyset$ by \eqref{eq:3:int}.

(ii)
It follows from 
(i)
and Remark \ref{sufamily:of:a:centered:family:is:centered}(i)
that $\{{\cal A}_n:n\in\N\}$ is a centered sequence in $X$. 
Since each $\A_n$ is a pseudobase of $X$, from
Definition 
\ref{def:nested:and:complete:sequences}(ii) we conclude that 
$\{{\cal A}_n:n\in\N\}$ is a complete sequence in $X$. 
\end{proof}

\begin{lemma}
\label{chec-Te}
\label{chec-SO}
A \v{C}ech-complete space is both \Telgarsky\ and \SOp.
\end{lemma}

\proof
Let $X$ be a \v{C}ech-complete space.
There exists a compact space $Y$ containing $X$ as a subspace and a family $\{V_n:n\in\N\}$ of open subsets of $Y$ such that $X=\bigcap\{V_n:n\in\N\}$; one can take the Stone-\v{C}ech compactification $\beta X$ of $X$ as $Y$. 
For each $n\in \N$, the family
$$
\cal A_n=\{U\cap X:U\textrm{ is a non-empty open set in } Y \mbox{ such that }\Cl_{Y}U\subseteq V_n\}
$$
is both
a base and a pseudobase in $X$ such that $\A_n\subseteq \B_n$, where $\B_n$ is as in \eqref{B:in:V_n}.
Thus, the sequence $\{{\cal A}_n:n\in\N\}$ is complete by 
Lemma \ref{centered:families:in:G-delta-sets}(ii), so $X$ is
\Telgarsky\ by Definition \ref{four:properties}(iv).

Similarly, for each $n\in \N$, the family
$$
\cal A_n=\{Z\cap X:Z\textrm{ is a zero-set in } Y \mbox{ such that }Z\subseteq V_n\mbox{ and }\Int_{Y}Z\not=\emptyset\}
$$
is a pseudobase in $X$ consisting of zero-sets such that $\A_n\subseteq \B_n$.
Arguing as above, we conclude that $\{{\cal A}_n:n\in\N\}$  is a complete sequence in $X$. Therefore, $X$ is \SOp\ 
by Definition \ref{four:properties}(iii).
\endproof

\begin{proposition}
\label{loc:psc:is:Telgarsky}
A locally pseudocompact space is strongly \Telgarsky.
\end{proposition}

\proof
Let $X$ be a locally pseudocompact space. Then the family
$\cal B=\{U\subseteq X:U$ is a non-empty  open subset of $X$ and $\Cl_X(U)$ is pseudocompact$\}$ is a base of $X$.
  
Suppose that $\{U_n:n\in \N\}\subseteq {\cal B}$ is a nested sequence.
Then
$\{U_n:n\in \N \}$ is a decreasing sequence of non-empty open sets in the 
pseudocompact space $\Cl_X(U_1)$. 
Therefore, $\bigcap \{\Cl_X(U_n):n\in \N \}\not=\emptyset$. 
By Remark \ref{rem:nested}, $\bigcap \{U_n:n\in \N \}\not=\emptyset$ as well. 
\endproof

The following diagram summarizes relations between the classical notions mentioned above.

\begin{center} \medskip\hspace{1em}\xymatrix{
\text{Locally compact}\ar[r]\ar[d]&  \text{\v{C}ech-complete}\ar[r]\ar[dd]&\text{S\'anchez-Okunev}\ar[d]&\\
\text{Locally pseudocompact}\ar[d]&	&\text{Todd}\ar[r]&\text{Baire}\\
\text{Strongly Telg\'arsky}\ar[r]&  \text{Telg\'arsky}\ar[r]&\text{Oxtoby}\ar[u]&
  }
\end{center}
\begin{center}
Diagram 1.
\end{center}

\section{Compactness-like properties}
\label{Compactness-like properties}

\begin{definition}
A family $\mathcal{F}$ of non-empty subsets of a set $X$ will be called {\em \cc } if every decreasing 
sequence $\{F_n:n\in\N\}\subseteq \mathcal{F}$ 
has non-empty intersection.
\end{definition}

\begin{remark}
(i) {\em A space $X$ is countably compact if and only if the family of all non-empty closed sets in $X$ is \cc\/}. 

(ii) {\em A space $X$ is pseudocompact if and only if the family of all non-empty zero-sets in $X$ is \cc\/}; see \cite[5H(4)]{GJ}.
\end{remark}

The relevance of the notion of a \cc\ family to completeness properties from Section \ref{section:1} can be seen from the following 
\begin{remark}
\label{compact:are:strongly:complete}
{\em Let $\B$ be a family of subsets of a topological space $X$, and let $\B_n=\B$ for every $n\in\N$.\/}

(i) {\em If $\B$ is \cc, then the sequence
$\{\B_n:n\in\N\}$ is centered.\/}
Indeed, let $\{B_n:n\in\N\}$ be a nested sequence such that $B_n\in\B_n$ for every $n\in\N$.
Then 
the sequence $\{B_n:n\in\N\}\subseteq \B$ is decreasing by Remark \ref{decreasing:vs:nested}(i).
Since $\B$ is \cc, $\bigcap_{n\in\N} B_n\not=\emptyset$. 

(ii) 
{\em If  $\B$ consists of clopen sets in $X$ and $\{\B_n:n\in\N\}$ is centered, then $\B$ is \cc.\/}
Indeed, let 
$\{B_n:n\in\N\}$ be a decreasing sequence such that $B_n\in\B$ for every $n\in\N$.
Then $\{B_n:n\in\N\}$ is nested by Remark \ref{decreasing:vs:nested}(ii).
Since $\{\B_n:n\in\N\}$ is centered, $\bigcap_{n\in\N} B_n\not=\emptyset$. 

(iii)
{\em If $\B$ is a \cc\ pseudobase of $X$, then 
$\{\B_n:n\in\N\}$ is a complete sequence in $X$.\/} 
This follows from 
item (i) and Definition \ref{def:nested:and:complete:sequences}(ii).
\end{remark}

Item (iii) of this 
remark suggests the following generalization of properties that appeared in Definition \ref{four:properties}.

\begin{definition} 
\label{def:compact:properties}
A topological space $X$ will be called:
\begin{itemize}
\item[(i)] {\em Todd \cc\/} if $X$ has a \cc\/ pseudobase;
\item[(ii)] {\em Oxtoby \cc\/} if $X$ has a \cc\/ pseudobase consisting of open sets;
\item[(iii)] {\em S\'anchez-Okunev \cc\/} if $X$ has a \cc\/ pseudobase which consists of zero-sets of $X$;
\item[(iv)] {\em Telg\'arsky \cc\/} if $X$ has a \cc\/ base.  
\end{itemize}
\end{definition}

The following is an immediate corollary of Remark \ref{compact:are:strongly:complete}(i) and Definition \ref{def:compact:properties}.

\begin{proposition}
\label{direct:implications}
\begin{itemize}
\item[(i)] A Todd \cc\ space is \sTp.
\item[(ii)] An  Oxtoby \cc\ space is \sOp.
\item[(iii)] A S\'anchez-Okunev \cc\ space is \sSOp.
\item[(iv)] A Telg\'arsky \cc\ space is \sTep.
\end{itemize}
\end{proposition}

In \cite{SO}, F. S\'anchez-Texis and O. Okunev essentially proved the following result (in our terminology): 
 
\begin{lemma}
\label{lemma:SOcc}
A $G_\delta$-dense subspace of a S\'anchez-Okunev \cc\ space is itself S\'anchez-Okunev \cc.
\end{lemma}

\proof
Let $Y$ be a $G_\delta $-dense subspace of a S\'anchez-Okunev \cc\ space $X$.
Let $\cal F$ be some \cc\/ pseudobase consisting of zero-sets in $X$.
Since $Y$ is dense in $X$, the family
$\cal G=\{F\cap Y:F\in \cal F\}$ is a pseudobase in $Y$. 
Clearly, $\cal G$ is a family of zero-sets in $Y$. 
Let 
$\{G_n:n\in\N\}\subseteq \cal G$ be a 
decreasing sequence.
For every $n\in N$, there is $F_n\in \cal F$ such that $F_n\cap Y=G_n$.

Let $n\in\N$. Since $Y$
is $G_\delta $-dense in $X$, \cite[Proposition 2.5]{SO} implies that
$F_{n+1}=\Cl_X(F_{n+1}\cap Y)=\Cl_X(G_{n+1})\subseteq\Cl_X(G_n)= \Cl_X(F_n\cap Y)= F_n$.

Hence, $\{F_n:n\in\N\}$ is a decreasing sequence of members of $\cal F$. 
Since $\cal F$ is \cc,
$G=\bigcap\{F_n:n\in\N\}$ is non-empty. 
 Since $G$ is a $G_\delta $-set in $X$ and $Y$ is $G_\delta$-dense in $X$,
we have $G\cap Y\not=\emptyset$, so
$\bigcap \{G_n:n\in\N\}=\bigcap \{F_n\cap Y:n\in\N\}=\bigcap \{F_n:n\in\N\}\cap Y=G\cap Y\not=\emptyset$.
\endproof

\begin{proposition}
\label{locally:compact:are:SOcc}
\label{weakly:psc:are:SO}
\begin{itemize}
\item[(i)]
A locally compact space is S\'anchez-Okunev \cc.
\item[(ii)]
A weakly pseudocompact space is S\'anchez-Okunev \cc.
\item[(iii)]
A locally pseudocompact space 
is 
S\'anchez-Okunev \cc.
\end{itemize}
\end{proposition}

\begin{proof}
(i) 
If $X$ is a locally compact space, then the family $\cal F=\{Z\subseteq X: Z$ 
is a compact zero-set in $X$ such that $\inte_XZ\not=\emptyset\}$
is a \cc\/ pseudobase in $X$. The second statement follows from 
Proposition \ref{direct:implications}(iii).

(ii)
Let $X$ be a weakly pseudocompact space. By Definition \ref{def:weakly:psc}(ii),
$X$ is $G_\delta $-dense in some compact space $Y$.
The latter is S\'anchez-Okunev \cc\ by 
item (i),
and the former is S\'anchez-Okunev \cc\ by Lemma \ref{lemma:SOcc}.

(iii)
Let $X$ be a locally pseudocompact space.
If $X$ is Lindel\"of, then $X$ is locally compact, so 
$X$ is S\'anchez-Okunev \cc\ by item (i).
If $X$ is not Lindel\"of, $X$ is weakly pseudocompact by \cite[Corollary 2.5]{DT}, so $X$ 
is S\'anchez-Okunev \cc\ by item (ii).
\end{proof}

\begin{example}
\label{R:example}
The real line $\bb R$ is S\'anchez-Okunev \cc\ by Proposition \ref{locally:compact:are:SOcc}
 but 
$\R$ is not weakly pseudocompact; see \cite[Corollary 3.8]{GG}. 
\end{example}

\begin{proposition}
\label{zero:dimensional:psc}
A locally pseudocompact zero-dimensional space 
is 
Telg\'arsky \cc.
\end{proposition}
\proof
Let $X$ be a locally pseudocompact zero-dimensional space. Then the family
$\cal B=\{U\subseteq X:U$ is a non-empty  clopen subset of $X$ and $U$ is pseudocompact$\}$ is a base of $X$.
Let
$\{U_n:n\in \N\}\subseteq {\cal B}$ 
be
a decreasing sequence.
Observe that this sequence is contained in the pseudocompact space $U_1$. 
Therefore, $\bigcap \{\Cl_X(U_n):n\in \N \}\not=\emptyset$.
Since each $U_n$ is clopen in $X$, this gives
$\bigcap \{U_n:n\in \N \}\not=\emptyset$ as well. 
\endproof

We finish this section with the following diagram summarizing the relations between the notions introduced in the first three sections.

\begin{center} \medskip\hspace{1em}\xymatrix{
\text{Locally pseudocompact}\ar@/_5pc/[dd]\ar[rrrd]&\text{Pseudocompact}\ar[rr]\ar[l]& &\text{Weakly pseudocompact}\ar[d]\\
\text{Telg\'arsky c.compact}\ar[r]\ar[d]&\text{Oxtoby c.compact} \ar[r]\ar[d]&\text{Todd c.compact}\ar[d] &\text{S\'anchez-Okunev c.compact}\ar[d]\ar[l]\\
\text{Strongly Telg\'arsky}\ar[r]\ar[d]&\text{Strongly Oxtoby} \ar[r]\ar[d]&\text{Strongly Todd}\ar[d] &\text{Strongly S\'anchez-Okunev}\ar[d]\ar[l]\\
\text{Telg\'arsky}\ar[r]&\text{Oxtoby}\ar[r] &\text{Todd}\ar[d] &\text{S\'anchez-Okunev}\ar[l]\\
& &\text{Baire} &\\
}
\end{center}
\begin{center}
Diagram 2.
\end{center}

\section{Main results}\label{results}

Our first theorem 
demonstrates
that
most completeness properties (except possibly the Telg\'arsky series of completeness properties) 
coincide for metric spaces:

\begin{theorem}
\label{coincidence:for:metric:spaces}
For a metric space $X$, the following conditions are equivalent:
\begin{itemize}
\item[(i)] $X$ is S\'anchez-Okunev \cc; 
\item[(ii)] $X$ is Oxtoby \cc; 
\item[(iii)] $X$ is Todd \cc;
\item[(iv)] $X$ is \sSOp; 
\item[(v)] $X$ is \sOp; 
\item[(vi)] $X$ is \sTp;
\item[(vii)] $X$ is \SOp; 
\item[(viii)] $X$ is \Op; 
\item[(ix)] $X$ is \Tp;
\item[(x)] $X$ contains a dense zero-dimensional completely metrizable subspace $Z$. 
\end{itemize}
\end{theorem}

The equivalence of items (viii) and (x) in this theorem is due to  
J.M. Aarts and D.J. Lutzer
\cite[Corollary 2.4]{AL}. The implication (ix)$\Rightarrow $(x) in this theorem was claimed without a proof in 
\cite[Proposition 6.4]{DRT}.
In Theorem \ref{T:gives:dense:complete} below we show that a subspace $Z$ from item (x) of the above theorem can also be chosen to have a countably compact base consisting of clopen subsets of $Z$, so that this $Z$ is simultaneously Telg\'arsky \cc\ and S\'anchez-Okunev \cc.

For metric groups, even more can be said:

\begin{theorem}
\label{metric:Todd:are:Cech-complete}
For a metric group $X$, the following two properties can be added 
to the list of equivalent conditions in Theorem 
\ref{coincidence:for:metric:spaces}.
\begin{itemize}
\item[(xi)] $X$ is \v{C}ech-complete;
\item[(xii)] $X$ is \Telgarsky.
\end{itemize}
\end{theorem}

Proofs of Theorems \ref{coincidence:for:metric:spaces} and 
\ref{metric:Todd:are:Cech-complete} are postponed until Section \ref{section:metric}.

\begin{definition}
\label{def:polish:factorizable}
A topological group $G$ will be called {\em Polish factorizable\/}
if 
for every continuous homomorphism $f:G\to  K$ to a separable metric group $K$, there exist a \v{C}ech-complete separable metric (aka Polish) group $H$ and continuous homomorphisms $\pi :G\to H$ and $h:H\to K$ such 
that $f=h\circ \pi$ and $H=\pi(G)$.
\end{definition}

Recall that a topological group $G$ is {\em precompact} if for every neighborhood $V$ of the identity $e$, there is a 
finite set $A\subseteq G$ such that $G=AV$.   
It is well known that a topological group $G$ is precompact if and only if it is topologically isomorphic to 
a subgroup of some compact group, or equivalently, if the Raikov completion $\varrho G$ of $G$ is compact. 

Most
completeness and 
compactness properties considered in this paper 
coincide for precompact groups:

\begin{theorem}
\label{equivalent:conditions:for:precompact}
For a precompact topological group $G$,
the following conditions are equivalent:
\begin{itemize}
\item[(i)] $G$ is pseudocompact;
\item[(ii)] $G$ is weakly pseudocompact;
\item[(iii)] $G$ is \SOp;
\item[(iv)] $G$ is \Op;
\item[(v)] $G$ is \Telgarsky; 
\item[(vi)] $G$ is \sSOp;
\item[(vii)] $G$ is \sOp;
\item[(viii)] $G$ is \sTep; 
\item[(ix)] $G$ is S\'anchez-Okunev \cc;
\item[(x)] $G$ is Polish factorizable.
\end{itemize}
\end{theorem}

The equivalence of  items (i) and (ii) in the previous theorem
confirms the validity of Conjecture \ref{conjecture:Tkachenko}.

An uncountable discrete space is weakly pseudocompact, as it is $G_\delta$-dense in its one-point compactification. Therefore, items (i) and (ii) in Theorem \ref{equivalent:conditions:for:precompact} are not equivalent without the assumption that the group $G$ is precompact.
Whether some of the equivalences from Theorem \ref{equivalent:conditions:for:precompact} remain valid for some  wider classes of topological groups remains unclear; see Question
\ref{omega-bounded:question}.

Theorem \ref{equivalent:conditions:for:precompact}
is proved in Section \ref{sec:polish;factorizable}.

\begin{corollary}
\label{zero:dimensional:precompact}
For a zero-dimensional precompact group $G$, the following two items can be added to the list of equivalent conditions of
Theorem \ref{equivalent:conditions:for:precompact}:
\begin{itemize}
\item[(i)] $G$ is Telg\'arsky \cc;
\item[(ii)] $G$ is Oxtoby \cc.
\end{itemize}
\end{corollary}

It is unclear if the assumption of zero-dimensionality can be dropped in this corollary; see Question \ref
{que:zero:dimensional:psc}.

\begin{theorem}
\label{equivalence:for:dense:subgroups:without:cech}
Let $G$ be a dense subgroup in the product $H=\prod_{i\in I} H_i $ of separable metric groups $H_i$. Then the 
following conditions are equivalent:
\begin{itemize}
\item[(i)] $G$ is \SOp;
\item[(ii)] $G$ is \Op; 
\item[(iii)] $G$ is \Telgarsky;
\item[(iv)] $G$ is \sSOp;
\item[(v)] $G$ is \sOp; 
\item[(vi)] $G$ is S\'anchez-Okunev \cc;
\item[(vii)] $G$ is Oxtoby \cc; 
\item[(viii)] $G$ is Polish factorizable;
\item[(ix)] $G$ is $G_\delta$-dense in $H$ and all $H_i$ are \v{C}ech-complete.
\end{itemize}
\end{theorem}

Item (ix) of this theorem says that \v{C}ech-completeness of all groups $H_i$ 
is a necessary condition for some dense subgroup $G$ of the product $H=\prod_{i\in I} H_i$ to have one of the equivalent completeness properties from items (i)--(viii) of this theorem. 

Proof of Theorem
\ref{equivalence:for:dense:subgroups:without:cech}
is postponed until Section \ref{sec:proofs}. 

\section{Applications to $C_p$-theory}
\label{C_p-section}

We are going to denote by $C_p(X,Y)$ the space of continuous functions 
from $X$ to $Y$ with the topology of the pointwise convergence, that is, the topology on $C_p(X,Y)$ is the topology generated by all sets of the form
$
\{f\in C(X,Y):f(x_i)\in V_i\ \text{ for }\  i=0,\ldots,n\},
$
where $n \in \N$, $x_0,\ldots,x_n\in X$ and $V_0,\ldots, V_n$ are open subsets of $Y$. When $Y$ is the real line $\mathbb{R}$ with its usual topology, we write $C_p(X)$ instead of $C_p(X,\mathbb{R})$. 

If $G$ is a group, take $f,g\in C_p(X,G)$ and define $(fg)(x)=f(x)g(x)$ and $(f^{-1})(x)=(f(x))^{-1}$.
With this operation, $C_p(X,G)$ is a subgroup of $G^X$, with the identity function $e$ defined by $e(x)=e_Y$ for every $x\in X$.  

For a topological group $G$, the function space $C_p(X,G)$ has been studied in 
\cite{SS} and recently in 
\cite{DRT}. The reader interested in $C_p$-theory can consult
the classical monograph of  V.V. Tkachuk \cite{Tka}. 

\begin{definitions}
Let $X$ and $Y$ be topological spaces.
\begin{enumerate}
\item A subspace $N$ of $X$ is {\em $C_Y$-embedded\/} in $X$ if every continuous function $f : N \to Y$ 
has a continuous extension over $X$.
\item A 
$X$
is {\em $b_Y$-discrete\/} if every countable subspace of $X$ is discrete and
$C_Y$-embedded in $X$.
\end{enumerate}
\end{definitions}

Our main result in $C_p$-theory is the following 

\begin{theorem}\label{Cp:theorem}
Let $G$ be a separable metric group and $X$ be a topological space such that $C_p(X,G)$ is 
dense in $G^X$. Then the following conditions are equivalent:
\begin{itemize}
\item[(i)] $C_p(X,G)$ is \SOp;
\item[(ii)] $C_p(X,G)$ is \Op; 
\item[(iii)] $C_p(X,G)$ is \Telgarsky;
\item[(iv)] $C_p(X,G)$ is \sSOp;
\item[(v)] $C_p(X,G)$ is \sOp; 
\item[(vi)] $C_p(X,G)$ is S\'anchez-Okunev \cc;
\item[(vii)] $C_p(X,G)$ is Oxtoby \cc; 
\item[(viii)] $C_p(X,G)$ is Polish factorizable;
\item [(ix)] $C_p(X,G)$ is $G_\delta $-dense in $G^X$ and $G$ is \v{C}ech-complete;
\item [(x)] $X$ is $b_G$-discrete and $G$ is \v{C}ech-complete.
\end{itemize}  
\end{theorem}

\proof
Since $G$ is first countable, $X$ is $b_G$-discrete if and only if 
$C_p(X,G)$ is $G_\delta $-dense in $G^X$ by \cite[Proposition 4.6]{DRT}. Hence, (ix) is equivalent to (x). 
Since $C_p(X,G)$ is a dense subgroup of $G^X$ by our assumption, the 
equivalence of items (i)--(ix) follows 
from Theorem \ref{equivalence:for:dense:subgroups:without:cech}.
\endproof

When $G=\mathbb{R}$, 
as a corollary we obtain the following classical result of V.V. Tkachuk:

\begin{corollary}\label{Theorem:4.1:Tk}\cite[Theorem 4.1]{Tk} The following conditions are equivalent for every space $X$: 
\begin{itemize}
\item[(i)] $X$ is $b_\R$-discrete;
\item[(ii)] $C_p(X)$ is \Op\/;
\item[(iii)] $\upsilon (C_p(X))=\mathbb{R}^{X}$, where $\upsilon Y$ denotes the Hewitt realcompactification of a space $Y$.
\end{itemize}
\end{corollary}
\proof
The equivalence (i)$\Leftrightarrow$(ii) follows from the equivalence of items (ii) and (x) in Theorem \ref{Cp:theorem} applied to $G=\mathbb{R}$.

Let us establish the equivalence (ii)$\Leftrightarrow$(iii).
Observe that 
$\R^X$ is a realcompact space; see \cite[Theorem 8.11]{GJ}. 
Since $C_p(X)$ is dense in $\R^X$, by 
\cite[Theorem 8.7]{GJ}, $\upsilon (C_p(X))=\R^X$ if and only if 
$C_p(X)$ is $C$-embedded in $\R^X$. By \cite[Corollary 6.3.15]{AT}, $\R^X$ is a Moscow space. 
By \cite[Theorems 6.1.4 and 6.1.7]{AT}, $C_p(X)$ is $C$-embedded in $\R^X$ if and only if 
$C_p(X)$ is $G_\delta $-dense in $\R^X$.
The equivalence (ii)$\Leftrightarrow$(ix) in Theorem \ref{Cp:theorem} applied to $G=\mathbb{R}$ implies that 
$C_p(X)$ is $G_\delta $-dense in $\R^X$ if and only if  $C_p(X)$ is \Op.
\endproof

It is worth emphasizing that not only Corollary \ref{Theorem:4.1:Tk} is not used in the proof of Theorem \ref{Cp:theorem}, but our proof 
of Theorem \ref{Cp:theorem} is totally different from the original proof of
 Corollary \ref{Theorem:4.1:Tk}. In fact, our proof is new even in the case $G=\mathbb{R}$.

Earlier result of D.J. Lutzer and R.A. McCoy \cite[Theorem 8.4]{LMc} also follows from Theorem \ref{Cp:theorem}.

The following corollary of Theorem \ref {Cp:theorem} seems somewhat surprising:
\begin{corollary}
Let $G$ be a separable metric group and $X$ be a topological space such that $C_p(X,G)$ is 
dense in $G^X$. If $C_p(X,G)$ has any of the equivalent properties (i)--(viii) from Theorem \ref {Cp:theorem}, then $G$ must be \v{C}ech-complete.
\end{corollary}

\begin{corollary}\label{Cp:corolary}
Let $G$ be a compact metric group and $X$ be a topological space such that $C_p(X,G)$ is 
dense in $G^X$. Then the following conditions are equivalent:
\begin{itemize}
\item[(i)] $C_p(X,G)$ is pseudocompact;
\item[(ii)] $C_p(X,G)$ is weakly pseudocompact;
\item[(iii)] $C_p(X,G)$ is \SOp;
\item[(iv)] $C_p(X,G)$ is \Op; 
\item[(v)] $C_p(X,G)$ is \Telgarsky;
\item[(vi)] $C_p(X,G)$ is \sSOp;
\item[(vii)] $C_p(X,G)$ is \sOp;
\item[(viii)] $C_p(X,G)$ is \sTep; 
\item[(ix)] $C_p(X,G)$ is S\'anchez-Okunev \cc;
\item[(x)] $C_p(X,G)$ is Oxtoby \cc; 
\item[(xi)] $C_p(X,G)$ is Polish factorizable;
\item [(xii)] $C_p(X,G)$ is $G_\delta $-dense in $G^X$;
\item [(xiii)] $X$ is $b_G$-discrete.
\end{itemize}  
\end{corollary}

Corollary \ref{Cp:corolary} 
strengthens substantially
Theorem 7.2 in \cite{DRT}. 

A host of conditions on $X$ which guarantee that $C_p(X,G)$ 
is dense in $G^X$ can be found in \cite{SS}.

\section{Completeness properties in metric spaces and groups: Proofs of Theorems \ref{coincidence:for:metric:spaces} and 
\ref{metric:Todd:are:Cech-complete}}
\label{section:metric}

Let $X$ be a metric space with a metric $d$. For a bounded subset $A$ of $X$,
let $\mathrm{diam} (A)=\sup\{d(a,b):a,b\in A\}$.
For a family $\A$ of subsets of 
$X$, we write
$\mathrm{diam}(\A)<\varepsilon$ when $\mathrm{diam}(A)<\varepsilon$ for all $A\in\A$.
We use $B_X(x,\varepsilon )$ to denote the ball in $X$ with the center $x\in X$
and the radius $\varepsilon >0$.

\begin{lemma}
\label{inscription:lemma}
Suppose that $X$ is a metric space, ${\cal U}$ is a pairwise disjoint family 
of non-empty open subsets of $X$ such that $\bigcup{\cal U}$ is dense in $X$,
$\B$ is a pseudobase of $X$ and $\delta>0$.
Then there exists $\A\subseteq \B$ such that:
\begin{itemize}
\item[(i)] $\Int_X (\A)=\{\Int_X(A):A\in\A\}$ is a pairwise disjoint family;
\item[(ii)] $\mathrm{diam}(\A)<\delta$;
\item[(iii)] for every $A\in\A$ one can find  $U\in {\cal U}$ such that
$\Cl_X(A)\subseteq U$;
\item[(iv)] $\bigcup\Int_X (\A)$ is dense in $X$.
\end{itemize}
\end{lemma}
\proof
Let $\A$ be a maximal subfamily of the family $\B$ having properties (i)--(iii); such a family exists by Zorn's lemma. 
It remains only to show that (iv) holds. Suppose the contrary.
Then $V=X\setminus\Cl_X(\bigcup\Int_X (\A))$ is a non-empty open subset of $X$.
Since $\bigcup{\cal U}$ is dense in $X$ by our assumption, 
$W=V\cap U\not=\emptyset$ for some $U\in{\cal U}$.
Choose $x_0\in W$. There exists $\varepsilon>0$ such that 
$B(x_0,\varepsilon)\subseteq W$.
Without loss of generality, we may assume that $\varepsilon<\delta$.
Since $\B$ is a pseudobase of $X$, 
$\Int_X(B)\subseteq B\subseteq B(x_0,\varepsilon/2)$ for some $B\in\B$.

We claim that the subfamily $\A^*=\A\cup\{B\}$ of $\B$ satisfies conditions (i)--(iii) as well. Indeed, $\Int_X(B)\subseteq W\subseteq V$, so from the choice of $V$ we conclude that $\Int_X(B)\cap \Int_X(A)=\emptyset$ for all $A\in\A$.
This establishes (i) for $\A^*$.
Since $B\subseteq B(x_0,\varepsilon/2)$, we have $\mathrm{diam}(B)\le \varepsilon<\delta$ and $\Cl_X(B)\subseteq B(x_0,\varepsilon)\subseteq W\subseteq U$. This proves both (ii) and (iii) for $\A^*$.

Since $B\not\in\A$, we have $\A\subseteq\A^*$ and $\A\neq\A^*$. This contradicts the maximality of $\A$.
\endproof

Our next theorem extends \cite[Corollary 2.4]{AL}
from \Op ness to \Tp ness.

\begin{theorem}
\label{T:gives:dense:complete}
Every  \Tp\ metric space $X$
contains a dense (zero-dimensional) completely metrizable subspace $Z$ having a \cc\ clopen base $\mathcal{C}$.
\end{theorem}

\proof
By Definition 
\ref{four:properties}(i), 
$X$ has a complete sequence
$\{\cal B_n:n\in \omega \}$.

By induction on $n\in\N$, we shall
define $\A_n\subseteq \B_n$ satisfying the following conditions:

\begin{itemize}
\item[(i$_n$)] $\Int_X (\A_n)$ is a pairwise disjoint family;
\item[(ii$_n$)] $\mathrm{diam}(\A_n)<1/(n+1)$;
\item[(iii$_n$)] if $n\ge 1$, then for every $A\in\A_n$ there exists $A'\in \A_{n-1}$ such that
$\Cl_X(A)\subseteq \Int_X(A')$;
\item[(iv$_n$)] $\bigcup\Int_X (\A_n)$ is dense in $X$.
\end{itemize}

First, we apply Lemma \ref{inscription:lemma} to ${\cal U}=\{X\}$, $\B=\B_0$
and $\delta=1/2$
to get $\A_0\subseteq \B_0$ satisfying (i$_0$)--(iv$_0$).
Assuming that $n\ge 1$ and $\A_{n-1}\subseteq \B_{n-1}$ satisfying 
(i$_{n-1}$)--(iv$_{n-1}$) has already been constructed,
we apply Lemma \ref{inscription:lemma} to ${\cal U}=\Int_X (\A_{n-1})$, $\B=\B_n$
and $\delta=1/(n+1)$
to get $\A_n\subseteq \B_n$ satisfying (i$_n$)--(iv$_n$).

For each $n\in \N$, the set  
\begin{equation}
\label{eq:W_n}
W_n=\bigcup\Int_X (\A_n)
\end{equation}
is open and dense in $X$ by (iv$_n$). 
Since $X$ is \Tp, $X$ is a Baire space, so
\begin{equation}
\label{eq:Z}
Z=\bigcap\{W_n:n\in \N \}
\end{equation}
is dense in $X$.
Define
\begin{equation}
\label{eq:O_A:C_A}
\ 
O_A=\Int_X(A)
\ 
\mbox{ and }
\ 
C_A=O_A\cap Z=\Int_X(A)\cap Z
\ 
\mbox{ for each }
\ 
A\in\bigcup_{n\in\N} \A_n.
\end{equation}
Let $K$ be a compact space containing $X$ as a dense subspace.
(For example, one can take $K$ as $K$.)
For every $A\in\bigcup_{n\in\N} \A_n$, since $O_A$ is an open subset of $X$,
we can fix an open subset $V_A$ of $K$ such that
\begin{equation}
\label{eq:V_A}
V_A\cap X=O_A=\Int_X(A),
\ \mbox{ and consequently, }
\
V_A\cap Z=C_A.
\end{equation}

\begin{claim}
The family
\begin{equation}
\label{family:C}
\mathcal{C}=\{C_A:A\in \bigcup_{n\in\N}\A_n\}.
\end{equation}
is a base of $Z$ consisting of non-empty clopen subsets of $Z$; in particular,
$Z$ is zero-dimensional.
\end{claim}
\proof
Let $V$ be an open subset of $Z$ and $z\in V$.
Choose an open subset $U$ of $X$ such that $U\cap Z=V$.
There exists $\varepsilon>0$ such that $B_X(z,\varepsilon)\subseteq U$.
Fix $n\in\N$ such that $1/(n+1)<\varepsilon$. 
Note that $x\in Z\subseteq W_n$ by \eqref{eq:Z}, so 
\eqref{eq:W_n} allows us to find $A\in\A_n$ such that 
$x\in \Int_X(A)$.
Now $\mathrm{diam}(\Int_X(A))\le\mathrm{diam}(A)\le \mathrm{diam}(\A_n)<1/(n+1)<\varepsilon$
by (ii$_n$).
Therefore, $x\in O_A=\Int_X(A)\subseteq B_X(x,\varepsilon)\subseteq U$.
Then
$x\in C_A=O_A\cap Z\subseteq U\cap Z=V$
by \eqref{eq:O_A:C_A}.
Since $C_A\in{\cal C}$ by \eqref{family:C}, this shows that ${\cal C}$ is a base of $Z$.

It remains only to note that all
members of ${\cal C}$ are 
clopen 
in $X$.
Indeed, let $n\in\N$ and $A\in\A_n$. Since $\Int_X(\A_n)$ is a family of pairwise disjoint open subsets of $X$ by (i$_n$),
$\Int_X(A)$ is a clopen subset of $W_n=\bigcup\Int_X(\A_n)$. 
Since $Z\subseteq W_n$ by \eqref{eq:Z}, the set $C_A=\Int_X(A)\cap Z$ is clopen in $Z$.
\endproof

\begin{claim}
\label{three:inclusions}
For $n\in \N$, $A_{n+1}\in\A_{n+1}$ and $A_n\in\A_n$, the following conditions are equivalent:
\begin{itemize}
\item[(i)] $\Cl_X({A_{n+1}})\subseteq \Int_X({A_n})$;
\item[(ii)] $C_{A_{n+1}}\subseteq C_{A_n}$;
\item[(iii)] $V_{A_{n+1}}\cap V_{A_n}\not=\emptyset$.
\end{itemize}
\end{claim}
\begin{proof}
(i)$\Rightarrow$(ii) Since $\Int_X({A_{n+1}})\subseteq \Cl_X({A_{n+1}})\subseteq \Int_X({A_n})$ by (i),
$C_{A_{n+1}}=\Int_X({A_{n+1}})\cap Z\subseteq \Int_X({A_n})\cap Z=C_{A_n}$
by \eqref{eq:O_A:C_A}.

(ii)$\Rightarrow$(iii) It follows from (ii) that
$\emptyset\not=C_{A_{n+1}}=C_{A_{n+1}}\cap C_{A_n}=(V_{A_{n+1}}\cap Z)\cap (V_{A_n}\cap Z)=V_{A_{n+1}}\cap V_{A_n}\cap Z\subseteq V_{A_{n+1}}\cap V_{A_n}$
by \eqref{eq:V_A}, which yields (iii).

(iii)$\Rightarrow$(i) 
Since 
$X$ 
is dense in $K$ 
and the set
$V_{A_{n+1}}\cap  V_{A_n}\not=\emptyset$
is open in $K$,
from \eqref{eq:V_A} we have 
\begin{equation}
\label{O:O'}
\emptyset\not=X\cap V_{A_{n+1}}\cap V_{A_n}
=
(X\cap V_{A_{n+1}})\cap  (X\cap V_{A_n})
=
O_{A_{n+1}}\cap  O_{A_n}
=
\Int_X(A_{n+1})\cap  \Int_X(A_n).
\end{equation}
By (iii$_{n+1}$), 
$\Cl_X(A_{n+1})\subseteq \Int_X(A')$ for some $A'\in \A_n$.
Combining this with \eqref{O:O'}, we conclude 
that $\Int_X(A')\cap \Int_X(A_n)\not=\emptyset$.
Since $A', A_n\in \A_n$, from (i$_n$) we get
$A'=A_n$.
This proves that 
$
\Cl_X(A_{n+1})\subseteq \Int_X(A_n).
$
\end{proof}

\begin{claim}
\label{two:inclusions}
If 
$\{A_n:n\in\N\}$ is a nested sequence such that
$A_n\in\A_n$ for every $n\in\N$,
then $\bigcap_{n\in\N} C_{A_n}\not=\emptyset$
and
$\bigcap_{n\in\N} V_{A_n}\subseteq Z$.
\end{claim}
\begin{proof}
The sequence $\{A_n:n\in\N\}$ is nested 
by our assumption. Since $A_n\in\A_n\subseteq \B_n$ and $\{\B_n:n\in\N\}$ is a complete sequence, 
$
\bigcap_{n\in\N} A_n\not=\emptyset
$
by Definition \ref{def:nested:and:complete:sequences}.
Since $O_{A_n}=\Int_X(A_n)$ by \eqref{eq:O_A:C_A}, 
from 
Remark \ref{rem:nested}
we conclude that 
$\bigcap_{n\in \N }O_{A_n}=\bigcap_{n\in\N} A_n$.
This allows us to
fix $x\in \bigcap_{n\in \N }O_{A_n}$.
Note that  $O_{A_n}\in\Int_X (\A_n)$ and \eqref{eq:W_n} imply $O_{A_n}\subseteq W_n$ for every $n\in\N$, so \eqref{eq:Z} gives
$x\in \bigcap_{n\in\N} W_n=Z$.
Now 
$x\in (\bigcap_{n\in \N }O_{A_n})\cap Z=
\bigcap_{n\in \N }(O_{A_n}\cap Z)
=
\bigcap_{n\in \N }C_{A_n}\not=\emptyset$ by \eqref{eq:O_A:C_A}.

Let us
to prove that $\bigcap_{n\in\N} V_{A_n}\subseteq Z$.
Let $y\in  \bigcap_{n\in\N} V_{A_n}$ be arbitrary.
Since $x\in Z$, it suffices to show that
$x=y$.
Assume that $x\not=y$.
Then there exists an open subset $U$ of $K$ such that
$x\in U$ and $y\not\in\Cl_{K} U$.
Then $U\cap X$ is an open neighbourhood of $x$ in $X$,
so 
there exists $\varepsilon>0$ such that $B_X(x,\varepsilon)\subseteq U$.
Since (ii$_n$) holds for every $n\in\N$,
there exists 
$n \in\N$ such that $\mathrm{diam} (\A_n)<\varepsilon$.
Since $x\in O_{A_n}\subseteq A_n\in \A_n$,
we have  $\mathrm{diam} (O_{A_n})\le \mathrm{diam} (A_n)<\varepsilon$, so
$O_{A_n}\subseteq B_X(x,\varepsilon)\subseteq U$. 
Since $X$ is dense in $K$ and $V_{A_n}$ is open in $K$,
we get 
$
\Cl_{K} (V_{A_n})
=
\Cl_{K} (V_{A_n}\cap X)
=
\Cl_{K} (O_{A_n})
\subseteq
\Cl_{K}(U).
$
Since $y\not\in \Cl_{K}(U)$,
this implies 
$y\not\in \Cl_{K} (V_{A_n})$,
in contradiction with $y\in O_{A_n}\subseteq V_{A_n}
\subseteq \Cl_{K}(V_{A_n})$.
\end{proof}

\begin{claim}
\label{non-empty:intersections}
If $A_n\in\A_n$ for every $n\in\N$ and 
$C_{A_0}\supseteq C_{A_1}\supseteq\ldots\supseteq C_{A_n}\supseteq C_{A_{n+1}}\supseteq\ldots$,
then 
$\bigcap_{n\in\N} C_{A_n}\not=\emptyset$.
\end{claim}
\begin{proof}
It follows from our assumption, Definition \ref{def:nested} and 
Claim \ref{three:inclusions} that the sequence $\{A_n:n\in\N\}$ is nested, so 
$\bigcap_{n\in\N} C_{A_n}\not=\emptyset$
by Claim \ref{two:inclusions}.
\end{proof}

\begin{claim}
\label{cl:3}
Suppose that $n_0,n_1\in\N$ and $A_{n_i}\in \A_{n_i}$ for $i=0,1$. 
If $C_{A_{n_1}}\subsetneqq C_{A_{n_0}}$, then:
\begin{itemize}
\item[(i)]  $n_0<n_1$;
\item[(ii)] there exists $A_m\in\A_m$ for every $m\in\N$ satisfying
$n_0<m<n_1$ such that $C_{A_{n_0}}\supseteq C_{A_{n_0+1}}\supseteq 
\ldots\supseteq C_{A_{n_1-1}}\supseteq C_{A_{n_1}}$.
\end{itemize}
\end{claim}
\begin{proof}
(i) 
First, note that $n_0\not=n_1$. Indeed, assume that $n_0=n_1=n$.
Then $A_{n_0},A_{n_1}\in\A_n$.
Since $\Int_X(A_{n_0})\cap Z=C_{n_0}\not=C_{n_1}=\Int_X(A_{n_1})\cap Z$ by 
\eqref{eq:O_A:C_A},
we get $A_{n_0}\not=A_{n_1}$, and so $\Int_X(A_{n_0})\cap\Int_X(A_{n_1})=\emptyset$ by (i$_n$).
Now 
$C_{n_0}\cap C_{n_1}\subseteq \Int_X(A_{n_0})\cap\Int_X(A_{n_1})=\emptyset$.
On the other hand, $\emptyset\not= C_{A_{n_1}}= C_{A_{n_0}}\cap C_{A_{n_1}}$, as $C_{A_{n_1}}\subseteq C_{A_{n_0}}$.
This contradiction shows that $n_0\not=n_1$.

Suppose now that $n_1<n_0$. Since (iii$_{n_0}$), (iii$_{n_0-1}$)$,\ldots$, 
(iii$_{n_1-1}$) hold, there exists $A'\in \A_{n_1}$ such that
$\Cl_X(A_{n_0})\subseteq \Int_X(A')$. 
Then 
\begin{equation}
\label{included:interiors}
\Int_X(A_{n_0})\subseteq \Cl_X(A_{n_0})\subseteq\Int_X(A').
\end{equation} 
On the other hand,
$\emptyset\not=\Int_X(A_{n_1})\cap Z=C_{A_{n_1}}\subseteq C_{A_{n_0}}=\Int_X(A_{n_0})\cap Z
\subseteq \Int_X(A_{n_0})$.
This shows that $\Int_X(A_{n_1})\cap \Int_X(A')\not=\emptyset$.
Since $A_{n_1},A'\in \A_{n_1}$ and $\Int_X(\A_{n_1})$ is pairwise disjoint by 
(i$_{n_1}$), 
this implies the equality $A_{n_1}=A'$.
Recalling \eqref{included:interiors}, we get
$C_{A_{n_0}}=\Int_X(A_{n_0})\cap Z\subseteq \Int_X(A_{n_1})\cap Z=C_{A_{n_1}}$.
Since the converse inclusion $C_{A_{n_1}}\subseteq C_{A_{n_0}}$ holds by our assumption,
we conclude that $C_{A_{n_0}}=C_{A_{n_1}}$, in contradiction with $C_{A_{n_0}}\not=C_{A_{n_1}}$.

(ii) We shall construct the required sequence $\{A_m:m\in\N, n_0<m<n_1\}$ by a (finite) reverse induction on $m$.
Assume that $n_0<m\le n_1$ and $A_m\in\A_m$ has already been chosen in such a way that $C_{A_{m}}\subseteq C_{A_{m+1}}$ (if $m\not=n_1$).
By (iii$_m$), there exists $A_{m-1}\in\A_{m-1}$ such that 
$\Cl_X(A_{m-1})\subseteq \Int_X(A_m)$.
Now $C_{A_m}\supseteq C_{A_{m-1}}$ by Claim \ref{three:inclusions}.

We have defined sets $A_{n_0+1}\in \A_{n_0+1}$, $A_{n_0+2}\in \A_{n_0+2},
\ldots, A_{n_1-1}\in \A_{n_1-1}$
such that
\begin{equation}
\label{chain:of:Cs}
C_{A_{n_0+1}}\supseteq C_{A_{n_0+2}}\supseteq\ldots\supseteq C_{A_{n_1-1}}
\supseteq 
C_{A_{n_1}}.
\end{equation}
Since $C_{A_{n_0}}\supseteq C_{A_{n_1}}$ by our assumption,
$\emptyset\not=C_{A_{n_1}}\subseteq C_{A_{n_0}}\cap C_{A_{n_0+1}}
\subseteq V_{A_{n_0}}\cap V_{A_{n_0+1}}$
by \eqref{eq:V_A} and \eqref{chain:of:Cs}, so $C_{A_{n_0}}\supseteq C_{A_{n_0+1}}$
by Claim \ref{three:inclusions}.
\end{proof}

\begin{claim}
The family ${\cal C}$ is \cc.
\end{claim}
\begin{proof}
Let $\{D_m:m\in\N\}\subseteq {\cal C}$ be a decreasing sequence.
We consider two cases.

{\em Case 1\/}. {\sl There exists $l\in\N$ such that
 $D_i=D_l$ for all integers $i\ge l$.\/}
 Then $\bigcap_{m\in\N} D_m=D_l\not=\emptyset$.
 
{\em Case 2\/}. {\sl There exists an increasing sequence
$\{m_k:k\in\N\}$ of integers such that 
$D_{m_k}=D_{m_k+1}=\ldots=D_{m_{k+1}-1}$ and $D_{m_k}\not=D_{m_{k+1}}$ for all $k\in\N$\/}. Since $\{D_k:k\in\N\}\subseteq {\cal C}$ is a decreasing sequence,
$\bigcap_{m\in\N} D_m=\bigcap_{k\in\N} D_{m_k}$. Therefore, by considering 
the subsequence $\{D_{m_k}:k\in\N\}$ of the sequence $\{D_m:m\in\N\}$, 
we may assume, without loss of generality, that the sequence $\{D_m:m\in\N\}$ itself is strictly decreasing, i.e.,
$D_{m+1}\not=D_m$ for all $m\in\N$.

For every $k\in\N$, 
from $D_k\in{\cal C}$ and \eqref{family:C}, we conclude that
$D_k=C_{A_{n_k}}$ for some $n_k\in\N$ and $A_{n_k}\in \A_{n_k}$.
By Claim \ref{cl:3}(i), the sequence $\{n_k:k\in\N\}$ is strictly increasing. 
By Claim \ref{cl:3}(ii), the sequence $\{A_{n_k}:k\in\N\}$ can be enlarged
to a sequence $\{A_n:n\in\N\}$ in such a way that the corresponding sequence
$\{C_{A_n}:n\in\N\}$ is (not necessarily strictly) decreasing, $\{A_{n_k}:k\in\N\}$ becomes a subsequence of the sequence $\{A_n:n\in\N\}$, and $A_n\in\A_n$ for all $n\in\N$.
Applying Claim \ref{non-empty:intersections},
we conclude that
$\bigcap_{n\in\N}C_{A_n}\not=\emptyset$.
Since 
$C_{A_{n_k}}=D_k$ for every $k\in\N$ and the sequence 
$\{C_{A_n}:n\in\N\}$ is decreasing, it follows that
$\bigcap_{k\in\N} D_k= \bigcap_{k\in\N} C_{A_{n_k}}=\bigcap_{n\in\N}C_{A_n}\not=\emptyset$.
\end{proof}

\begin{claim}
$Z$ is \v{C}ech-complete. 
\end{claim}
\begin{proof}
Let $n\in\N$. Since $V_A$ is an open subset of $K$ for every $A\in\A_n$, the set 
\begin{equation}
\label{eq:V_n}
V_n=\bigcup_{A\in\A_n} V_A
\end{equation}
is open in $K$.
Therefore, 
$\bigcap_{n\in\N} V_n$
is a $G_\delta$-subset of the compact space $K$, so it is \v{C}ech-complete.
Therefore, it suffices to show that 
$Z=\bigcap_{n\in\N} V_n$.

Let $z\in Z$ be arbitrary. For each $n\in \N$, we can use  \eqref{eq:W_n} and \eqref{eq:Z} to fix 
$A_n\in\A_n$ such that 
$z\in O_{A_n}=\Int_X(A_n)\in \Int_X (\A_n)$.
Note that
$O_{A_n}=V_{A_n}\cap Z\subseteq V_{A_n}$.
Since $A_n\in\A_n$, we have $V_{A_n}\subseteq V_n$ by 
\eqref{eq:V_n}. This shows that $z\in V_n$.
Since this inclusion holds for each $n\in\N$, we get
$z\in\bigcap_{n\in\N} V_n$. 
This establishes the inclusion
$Z\subseteq \bigcap_{n\in\N} V_n$. 

To check the inverse inclusion 
$\bigcap_{n\in\N} V_n\subseteq Z$, fix 
$y\in \bigcap_{n\in\N} V_n$.
We are going to show that $y\in Z$.
For each $n\in\N$ use $y\in V_n$ and \eqref{eq:V_n}
to choose $A_n\in\A_n$ such that 
$y\in V_{A_n}$.
Then $y\in V_{A_{n+1}}\cap V_{A_n}\not=\emptyset$, so 
$\Cl_X(A_{n+1})\subseteq \Int_X(A_n)$ by the implication 
(iii)$\Rightarrow$(i) of  
Claim \ref{three:inclusions}.
Therefore, the sequence $\{A_n:n\in\N\}$ is nested by 
Definition \ref{def:nested}.
Now $y\in\bigcap_{n\in\N} V_{A_n}\subseteq Z$ by Claim
\ref{two:inclusions}.
\end{proof}

\begin{lemma}
\label{extention:of:cc:families}
Let $Z$ be a subspace of a topological space $X$
and ${\cal C}$ a family of non-empty subsets of $Z$.
For every $C\in{\cal C}$ let $F_C$ be a subset of $X$ such that
$F_C\cap Z=C$. If the family ${\cal C}$ is \cc, then so is the family ${\cal F}=\{F_C:C\in{\cal C}\}$.
\end{lemma}
\begin{proof}
Let $\{C_n:n\in\N\}\subseteq {\cal C}$ be a sequence such that $\{F_{C_n}:n\in\N\}\subseteq {\cal F}$ 
is a decreasing sequence; that is,
$F_{C_{n+1}}\subseteq F_{C_n}$ for every
$n\in N$. Then
$C_{n+1}=F_{C_{n+1}}\cap Z\subseteq  F_{C_n}\cap Z=C_n$
by the definition of $F_C$.
This means that 
$\{C_n:n\in\N\}$ is a decreasing sequence in $\cal C$. 
Since this family is \cc, 
$\emptyset \not= \bigcap \{C_n:n\in\N\}\subseteq \bigcap \{F_{C_n}:n\in\N\}$.
\end{proof}

\begin{lemma}
\label{dense_strongly_sanchez_okunev}
Let 
$Z$ be a 
a dense subspace of a topological space $X$ 
and let ${\cal C}$ be a \cc\ pseudobase of $Z$.
Then:
\begin{itemize}
\item[(i)]
if all members of ${\cal C}$ are open in 
$Z$, then 
${\cal U}=\{U_C: C\in{\cal C}\}$ 
is a \cc\ pseudobase of $X$ consisting of open subsets of $X$,
where $U_C$ is 
any
open subset of $X$ such that $U_C\cap Z=C$;
  \item[(ii)] if all members of ${\cal C}$ are closed in $Z$, then
  ${\cal F}=\{\Cl_X(C):C\in{\cal C}\}$ is a \cc\ pseudobase
  of $X$  consisting of closed subsets of $X$.  
\end{itemize}
\end{lemma}
\begin{proof}
A straightforward check that 
${\cal U}$
is a pseudobase of $X$ 
is left to the reader.
Now
 ${\cal U}$ is \cc\ by Lemma \ref{extention:of:cc:families}.

(ii) 
A straightforward check that 
${\cal F}$
is a pseudobase of $X$ 
is left to the reader.
Note that $\Cl_X(C)\cap Z=\Cl_Z(C)=C$ for every $C\in{\cal C}$, 
so ${\cal F}$ is \cc\ by Lemma \ref{extention:of:cc:families}.
\end{proof}

\medskip
\noindent
{\bf Proof of Theorem \ref{coincidence:for:metric:spaces}:}
The implication (ix)$\Rightarrow$(x) follows from 
Theorem \ref{T:gives:dense:complete}.

(x)$\Rightarrow$(i)  and (x)$\Rightarrow$(ii). Suppose that $X$ contains 
a dense subspace $Z$ with a \cc\ base $\mathcal{C}$ 
consisting of clopen subsets of $Z$. 
By Lemma \ref{dense_strongly_sanchez_okunev}(i),
${\cal U}$ is  a \cc\ pseudobase of $X$ consisting of open subsets of $X$, so $X$ is Oxtoby \cc\ by Definition \ref{def:compact:properties}(ii).
By Lemma \ref{dense_strongly_sanchez_okunev}(ii),
${\cal F}$ is a \cc\ pseudobase
  of $X$  consisting of closed subsets of $X$.
Since $X$ is perfectly normal, every member of ${\cal F}$ is a zero-set in $X$; see \cite[1.5.19(iii)]{En}. 
Therefore, $X$ is S\'anchez-Okunev \cc\ by 
Definition  \ref{def:compact:properties}(iii).

The rest of the implications follows from Diagram 2. 
\endproof

\begin{lemma}
Let $G$ be a
topological group with the Baire property. If $X$ is a dense $G_\delta $-subspace of $G$ then $G=XX^{-1}$.
\end{lemma}

\proof
Let $g\in G$ be arbitrary. 
Since
$X$ is a dense $G_\delta $-subset of $G$
and the map $x\mapsto gx$ ($x\in G$) is a homeomorphism of $G$ onto itself, $Y=gX$ is a dense $G_\delta $-subset of $G$.
Since $G$ has the Baire property,
$H=X\cap Y\not=\emptyset $, so
we can 
take $h\in H$.
There is $x\in X$ such that $h=gx$. Therefore, $g=hx^{-1}\in XX^{-1}$.  
\endproof

The following corollary is well known; see, for example, \cite[Lemma 4.2]{Tk} where it is mentioned without a proof.

\begin{corollary}\label{dense_cech}
If $H$ is a dense subgroup of a topological
group $G$ such that $H$ contains a dense 
\v{C}ech-complete subspace, then $H=G$. 
\end{corollary}

\proof
Let $X$ be a dense \v{C}ech-complete subspace of $H$. 
Since $H$ is dense in $G$, so $X$ is.
Since $X$ has the Baire property, $G$ also has it. 
Since $X$ is  \v{C}ech-complete, it is
a $G_\delta $-set in $G$. 
By the previous lemma, $G=XX^{-1}$. Since $H$ is a subgroup of $G$
and $X\subseteq H$, we have $G=XX^{-1}\subseteq H\subseteq G$, so $H=G$.
\endproof

\medskip
\noindent
{\bf Proof of Theorem \ref{metric:Todd:are:Cech-complete}:}
(x)$\Rightarrow$(xi) 
Let $G=\varrho X$ be the Raikov completion of $X$. Then $G$ is a metric Raikov complete topological group. 
By \cite[Proposition 4.3.8]{AT}, $G$ is \v{C}ech-complete. 
By Corollary 
\ref{dense_cech}, $X=G$. In particular, $X$ is \v{C}ech-complete.

By Lemma \ref{chec-SO}, (xi) implies (xii),  and it is clear that (xii) implies 
(v). 
The rest of implications follow from Theorem \ref{coincidence:for:metric:spaces}.
\endproof

\section{Technical lemmas}
\label{technical:lemmas}

The proof of the next lemma follows closely that of \cite[S. 470]{Tka}.
\begin{lemma}
\label{lemma:Tkacuk}
Let $\{\cal A_n:n\in\N\}$ be a complete sequence in a space $X$. Then there exists a 
sequence $\{\B_{n}:n\in\N\}$ of pseudobases of $X$ 
 such that:
\begin{itemize}
\item[(i)] 
$\B_n\subseteq {\cal A}_n$ for every $n\in\N$;
\item[(ii)] If $m\in\N$ and
 $\{B_{n}:n\geq m\}$ is a nested sequence in $X$ such that $B_{n}\in\cal B_{n}$ for each $n\geq m$,
then the intersection  $\bigcap \{B_{n}:n\geq m\}$ is non-empty. 
\end{itemize}
\end{lemma}

\proof
If ${\cal A}$ and $\B$ are families of subsets of a topological space $X$, then symbol
${\cal A}\sqsubseteq \B$
means that for every $A\in {\cal A}$ there exists $B\in\B$ such that $\Cl_X(A)\subseteq \Int_X(B)$.

By induction on $n\in\N$, we define a pseudobase $\B_n\subseteq {\cal A}_n$ of $X$ as follows.
First, let $\B_0=\A_0$.
If $n\in\N$ and 
$\B_n\subseteq {\cal A}_n$ has already been defined,
then we let 
\begin{equation}
\B_{n+1}=\{A\in\A_{n+1}:\text{ there exists }B\in\B_{n}\text{ such that }
\Cl_X(A) \subseteq \Int_X(B)\}.
\end{equation}
By definition, $\B_{n+1}\sqsubseteq \B_{n}$ and $\B_{n+1}\subseteq \A_{n+1}$. 
Since $\B_{n}\subseteq \A_n$ and both $\A_{n+1}$ and $\A_n$ are 
pseudobases of $X$, one easily sees that $\B_{n+1}$ is a pseudobase of $X$ as well.

We have constructed a sequence of pseudobases of $X$ such that (i) holds
and
$\B_{n+1}\sqsubseteq \B_{n}$ for every $n\in\N$.

(ii) Suppose that $m\in\N$ and
 $\{B_{n}:n\geq m\}$ is a (possibly) ``truncated'' nested sequence in $X$ such that $B_{n}\in\cal B_{n}$ for each $n\geq m$.
We are going to extend this sequence
to a ``full-size''  nested sequence $\{B_n:n\in\N\}$ in $X$
such that $B_n\in\B_n$ for all $n\in\N$.
If $m=0$, then there is nothing  to do.
Now, suppose that $m>0$. 
Since $\B_m\sqsubseteq \B_{m-1}\sqsubseteq\dots\sqsubseteq
\B_1\sqsubseteq \B_0$, we can 
construct, by a finite reverse induction on $m$, a sequence 
$B_{{m}-1}$, $B_{{m}-2},\ldots, B_{0}$
 such that $B_{i}\in\B_i$ and $\Cl_X(B_{i+1})\subseteq \Int_X(B_{i})$ for every 
$i= {m}-1, {m}-2,\ldots, 0$.
Clearly, $\{B_n:n\in\N\}$ is a nested sequence in $X$
extending the original ``truncated'' nested sequence $\{B_{n}:n\geq m\}$ and satisfying $B_n\in\B_n$ for all $n\in\N$.
From this and  (i), we conclude that $B_n\in\A_n$ for all $n\in\N$. 
Since the sequence $\{\cal A_n:n\in\N\}$ is complete, $\emptyset \not=\bigcap \{B_n:n\in\N\}\subseteq \{B_{n}:n\geq m\}$. 
This proves (ii). 
\endproof

The next corollary shows that, in the definition of a complete sequence $\{\B_n:n\in\N\}$, one may assume, without loss of generality, that $X\in \B_n$ for every $n\in\N$.

\begin{corollary}
\label{no:loss:of:generality}
If $\{\A_n:n\in\N\}$ is a complete sequence in a non-empty space $X$, then one can choose 
$\B_n\subseteq \A_n$ for every $n\in\N$ in such a way that 
$\{{\cal C}_n:n\in\N\}$ becomes a complete sequence in $X$, where
${\cal C}_n=\B_n\cup\{X\}$ for every $n\in\N$.
\end{corollary}
\begin{proof}
Use Lemma \ref{lemma:Tkacuk} to define the sequence $\{\B_n:n\in\N\}$ as in the conclusion of that lemma.
Let ${\cal C}_n=\B_n\cup\{X\}$ for $n\in\N$.
Let $\{B_n:n\in\N\}$ be a nested sequence in $X$ such that
$B_n\in{\cal C}_n$ for all $n\in\N$.
If $B_n=X$ for all $n\in\N$, then $\bigcap_{n\in\N} B_n=X\not=\emptyset$. Otherwise,
there exists $m\in\N$ such that $B_m\not=X$ and $B_0=B_1=\dots=B_{m-1}=X$.
Since the sequence $\{B_n:n\in\N\}$ is nested, 
it is decreasing by Remark \ref{decreasing:vs:nested}(i).
Since $B_m\not=X$,
it follows that $B_n\not=X$ (equivalently, $B_n\in\B_n$) for all $n\ge m$. 
By item (ii) of Lemma \ref{lemma:Tkacuk},
$\emptyset \not=\{B_{n}:n\geq m\}=\bigcap \{B_n:n\in\N\}$. 
\end{proof}

\begin{lemma}
\label{dense:nested:sequences}
Let $\{A_n:n\in\N\}$ and $\{B_n:n\in\N\}$  be two families of subsets of a topological space $X$ such that:
\begin{itemize}
\item[(i)]
 $B_n$ is a dense subset of $A_n$ for every $n\in\N$;
\item[(ii)] $\{B_n:n\in\N\}$ is nested.
\end{itemize}
Then $\{A_n:n\in\N\}$ is nested 
and $\bigcap\{A_n:n\in\N\}=\bigcap\{B_n:n\in\N\}$.
\end{lemma}
\proof
Fix $n\in\N$. Then $\Cl_X (B_{n+1}) \subseteq \Int _X (B_n)$
by (ii) and Definition \ref{def:nested}.
Furthermore,
$
\Cl_X (B_{n+1})=\Cl_X (A_{n+1})
$
by (i).
Finally, 
$B_n\subseteq A_n$ implies
$\Int _X (B_n)\subseteq\Int _X (A_n)$.
This shows that
$\Cl_X (A_{n+1})\subseteq \Int _X (A_n)$.
Since this inclusion holds for every $n\in\N$, the family
$\{A_n:n\in\N\}$ is nested by Definition \ref{def:nested}.

Since both $\{A_n:n\in\N\}$ and $\{B_n:n\in\N\}$ are nested, 
from 
Remark \ref{rem:nested}
  and (i) we get
$$
\bigcap\{A_n:n\in\N\}=\bigcap\{\Cl_X(A_n):n\in\N\}=
\bigcap\{\Cl_X(B_n):n\in\N\}=\bigcap\{B_n:n\in\N\}.
$$
This finishes the proof.
\endproof

\begin{lemma}
\label{dense:subfamily}
Let $\{{\cal A}_n:n\in\N\}$
and $\{{\cal B}_n:n\in\N\}$ be two sequences of families of subsets of a topological space $X$ having the following property:
For every $n\in\N$ and each $B\in {\cal B}_n$ there exists some $A\in {\cal A}_n$ such that $B$ is a dense subset of $A$.
If the sequence $\{{\cal A}_n:n\in\N\}$ is centered, then so is the
sequence $\{{\cal B}_n:n\in\N\}$.
\end{lemma}

\proof
Let $\{B_n:n\in\N\}$ be a nested sequence such that $B_n \in \cal{B}_n$ for all $n\in\mathbb{N}$.
By our assumption, for each $n\in\N$ we can find $A_n\in \cal{A}_n$
such that $B_n$ is a dense subset of $A_n$.
Applying Lemma \ref{dense:nested:sequences}, we conclude that 
the sequence $\{A_n:n\in\N\}$ is nested
and
\begin{equation}
\label{A_n:B_n}
\bigcap\{A_n:n\in\N\}=\bigcap\{B_n:n\in\N\}.
\end{equation}
Assume now that the sequence $\{{\cal A}_n:n\in\N\}$ is centered.
Then 
$\bigcap\{A_n:n\in\N\}\not=\emptyset$
by Definition \ref{def:nested:and:complete:sequences}(i).
Combining this with \eqref{A_n:B_n}, we conclude that 
$\bigcap\{B_n:n\in\N\}\not=\emptyset$.
Since this holds for each nested sequence $\{B_n:n\in\N\}$ such that
$B_n\in{\cal B}_n$ for all $n\in\N$, 
the sequence $\{{\cal B}_n:n\in\N\}$ is centered by
Definition \ref{def:nested:and:complete:sequences}(i).
\endproof

\begin{corollary}
\label{dense:sequences}
Let $\{{\cal A}_n:n\in\N\}$ be a complete sequence in a topological space $X$.
For every $n\in\N$ and each $A\in {\cal A}_n$, let 
$U_{n,A}$ be a dense subset of $A$ which is also open in $X$.
Then $\{{\cal B}_n:n\in\N\}$ is a complete sequence in $X$, where 
${\cal B}_n=\{U_{n,A}:A\in {\cal A}_n\}$
for each $n\in \N$.
\end{corollary}

\proof
Since $\{{\cal A}_n:n\in\N\}$ be a complete sequence in
$X$, it is centered by Definition \ref{def:nested:and:complete:sequences}(ii).
Applying Lemma \ref{dense:subfamily}, we conclude that 
the sequence $\{{\cal B}_n:n\in\N\}$ is centered as well. 
By Definition \ref{def:nested:and:complete:sequences}(ii),
it remains only to check that each ${\cal B}_n$
is a pseudobase for $X$.
To achieve this, we need to check both conditions from  
Definition \ref{def:pi-pseudobase}.

To check condition (i), fix $B\in {\cal B}_n$ arbitrarily.
Then $B=U_{n,A}$ for some $A\in {\cal A}_n$.
Since ${\cal A}_n$ is a pseudobase for $X$,
$\Int_X(A)\not=\emptyset$ by Definition \ref{def:pi-pseudobase}(i).
In particular, $A\not=\emptyset$.
Since $U_{n,A}$ is dense in $A\not=\emptyset$, 
it follows that $U_{n,A}\not=\emptyset$.
Since $U_{n,A}$ is open in $X$, $\emptyset\not=U_{n,A}=\Int_X(U_{n,A})=\Int_X(B)$.

To check (ii), take an arbitrary non-empty open subset $W$ of $X$.
Since ${\cal A}_n$ is a pseudobase for $X$, Definition \ref{def:pi-pseudobase}(ii) guarantees the existence of $A\in {\cal A}_n$ such that
$A\subseteq W$. Now $U_{n,A}\subseteq A\subseteq W$ and 
$U_{n,A}\in {\cal B}_n$.
\endproof

\begin{notation}
\label{good:sets}
For a map $f:X\to Y$ between topological spaces $X$ and $Y$, define
$${\cal P}_f=\{F\subseteq X:F=f^{\leftarrow} (f(F))\}.$$ 
\end{notation}

\begin{lemma}
\label{image:of:a:centered:sequence}
Let $f:X\to Y$ be a continuous map.
\begin{itemize}
\item[(i)] If $A,B\in {\cal P}_f$ and $\Cl_{f(X)}(A)\subseteq \Int_{f(X)}(B)$,
then $\Cl_X(A)\subseteq \Int_X(B)$.
\item[(ii)] If $\{B_n: n\in\N\}\subseteq {\cal P}_f$ 
and the sequence $\{f(B_n):n\in\N\}$ is nested in $f(X)$, 
then the sequence $\{B_n:n\in\N\}$ is nested in $X$.
\item[(iii)] If $\{{\cal B}_n: n\in\N\}$ is a centered sequence in $X$ such that $\bigcup \{{\cal B}_n: n\in\N\}\subseteq {\cal P}_f$, then
 $\{{\cal C}_n: n\in\N\}$ is a centered sequence in $f(X)$, where 
${\cal C}_n=\{f(B): B\in{\cal B}_n\}$
for every $n\in\N$.
\end{itemize}
\end{lemma}

\proof
(i)
Since $A\in {\cal P}_f$, we have
$A=f^{\leftarrow}(f(A))\subseteq f^{\leftarrow}(\Cl_{f(X)}(f(A)))$.
Since $f$ is continuous, the set $f^{\leftarrow}(\Cl_{f(X)}(f(A)))$ is closed in $X$, so 
\begin{equation}
\label{eq:2}
\Cl_X(A)\subseteq f^{\leftarrow}(\Cl_{f(X)}(f(A))).
\end{equation}
Since $\Cl_{f(X)}(A)\subseteq \Int_{f(X)}(B)$, one has
\begin{equation}
\label{eq:1}
f^{\leftarrow}(\Cl_{f(X)}(f(A)))\subseteq f^{\leftarrow}(\Int_{f(X)}(f(B))).
\end{equation}
Since $B\in {\cal P}_f$, we have
$f^{\leftarrow}(\Int_{f(X)}(f(B)))\subseteq f^{\leftarrow}(f(B))=B$.
Since $f$ is continuous, the set $f^{\leftarrow}(\Int_{f(X)}(f(B)))$ is open n $X$, so 
\begin{equation}
\label{eq:3}
f^{\leftarrow}(\Int_{f(X)}(f(B)))\subseteq \Int_X(B).
\end{equation}
Combining \eqref{eq:2}, \eqref{eq:1} and \eqref{eq:3} gives the conclusion of item (i).

(ii) follows from (i) and Definition \ref{def:nested}.

(iii) Let $\{C_n:n\in\N\}$ be a nested sequence in $f(X)$ such that $C_n\in{\cal C}_n$ for every $n\in\N$.  
For every $n\in\N$ fix $B_n\in{\cal B}_n$ such that $C_n=f(B_n)$.
By item (ii), the sequence $\{B_n:n\in\N\}$ is nested in $X$.
Since the sequence $\{{\cal B}_n:n\in\N\}$ is centered in $X$, Definition \ref{def:nested:and:complete:sequences}(i) implies that $\bigcap_{n\in\N}B_n \not= \emptyset$.
Therefore,
$\emptyset\not=\bigcap_{n\in\N} f(B_n)=\bigcap_{n\in\N} C_n$.
By Definition \ref{def:nested:and:complete:sequences}(i), this shows that 
the sequence $\{{\cal C}_n:n\in\N\}$ is centered in $f(X)$.
\endproof

\begin{lemma}
\label{inverse:preimages:of:good:sets}
Let $f:X\to f(X)$, $g:X\to g(X)$ and $h: g(X)\to f(X)$ be functions such that
$f=h\circ g$. If $P\in{\cal P}_f$, then 
$g(P)= h^{\leftarrow}(f(P))$.
\end{lemma}
\proof
Clearly,
$g(P)\subseteq h^{\leftarrow}(f(P))$.
Let us check the inverse inclusion $h^{\leftarrow}( f(P))\subseteq g(P)$.
Let $y\in h^{\leftarrow}( f(P))$ be arbitrary.
Since $y\in g(X)$, there exists $x\in X$ such that $g(x)=y$, and so
$f(x)=h(g(x))=h(y)\in f(P)$.
Therefore, $x\in f^{\leftarrow}(f(P))=P$, as $P\in{\cal P}_f$.
This means that $y=g(x)\in g(P)$.
\endproof

\begin{notation}
\label{non-empty-interior:sets}
For a map $f:X\to Y$ between topological spaces $X$ and $Y$, define
$$
\cal I_f=\{P\subseteq X:\Int_{f(X)}(f(F))\not=\emptyset \mbox{ for some set }F\subseteq P \mbox{ such that } F\in{\cal P}_f\}.
$$
\end{notation}

\begin{lemma}
\label{upwards:directed}
Let $f:X\to f(X)$, $g:X\to g(X)$ and $h:g(X)\to f(X)$
be continuous functions such that $f=h\circ g$ (we shall write
$f\preceq g$ in such a case). Then:
\begin{itemize}
\item[(i)]
${\cal P}_f\subseteq \cal{P}_g$;
\item[(ii)] 
${\cal I}_f\subseteq  \cal{I}_g$.
\end{itemize}
\end{lemma}
\proof
(i) 
Fix $P\in \cal P_f$. 
From Lemma \ref{inverse:preimages:of:good:sets} we get 
$$
g^{\leftarrow}(g(P))=
g^{\leftarrow}(h^{\leftarrow}(f(P)))
=
(h\circ g)^{\leftarrow}(f(P))
=
f^{\leftarrow}(f(P))
=P,
$$
where the last equality holds because $P\in{\cal P}_f$.
Therefore,
$P\in{\cal P}_g$.

(ii)
Assume now that  $P\in  {\cal I}_f$.
Then 
$U=\Int_{f(X)}(f(F))\not=\emptyset$ for some set $F\subseteq P$ such that $F\in{\cal P}_f$.
Since $h$ is a continuous surjection,
$V=h^{\leftarrow}(U)$ is a non-empty open subset of $g(X)$.
Since $U\subseteq f(F)$, from Lemma \ref{inverse:preimages:of:good:sets} we conclude 
that $V\subseteq g(F)$. Therefore, 
$\Int_{g(X)} (g(F))\not=\emptyset$.
Furthermore,
$F\in{\cal P}_g$ by item (i) applied to $F$ instead of $P$.
This shows that 
$P\in\cal{I}_g$.
\endproof

\section{A factorization theorem for \Tp ness}
\label{factorization:section}

For a topological space $X$ let
$\mathbb{S}_X$ be the set of all continuous mappings from $X$ into
the Hilbert cube $[0,1]^\N$. 
For $f,g\in\mathbb{S}_X$ we write
$f\preceq g$ if there exists a continuous mapping $h:g(X)\rightarrow
f(X)$ such that $f=h\circ g$.
  One can easily check that the relation
$\preceq$ on $\mathbb{S}_X$ is reflexive and transitive. However,
$\preceq$ is not  anti-symmetric. To fix this, we will introduce an
appropriate quotient of $\mathbb{S}_X$. If $f,g\in \mathbb{S}_X$,
$f\preceq g$ and $g\preceq f$, then we write $f\approx g$. Clearly,
$\approx$ is an equivalence relation on $\mathbb{S}_X$. One can
easily see that $f\approx g$ if and only iff there exists a
homeomorphism $h: g(X)\to f(X)$ such that $f=h\circ g$.

As usual, $[g]_\approx=\{f\in\mathbb{S}_X:f\approx g\}$ denotes the equivalence
class of $g\in  \mathbb{S}_X$ with respect to the relation
$\approx$.
Define
$\mathbb{F}_X=\{[g]_\approx: g\in \mathbb{S}_X\}$.
We write
$[f]_\approx\preceq[g]_\approx$ if $f\preceq g$. Clearly the
relation "$\preceq$" on $\mathbb{F}_X$ is well defined and makes
$\mathbb{F}_X$ into a partially ordered set (poset).
With a certain abuse of notation, from now on we will not distinguish
between $f\in \mathbb{S}_X$ and $[f]_\approx\in \mathbb{F}_X$.
In particular, we will write $f\in \mathbb{F}_X$ instead of cumbersome
$[f]_\approx\in \mathbb{F}_X$.

Suppose that $f:X\to f(X)$ is an arbitrary continuous mapping such
that $f(X)$ is separable metrizable. Since Hilbert cube is universal
for all separable metrizable spaces, there exists a homeomorphic
embedding $h: f(X)\to [0,1]^\N$. Now $g=h\circ f\in
\mathbb{S}_X$ (and with abuse of notation we have agreed upon also
$g\in \mathbb{F}_X$). So $f$ can be identified with its
``representative'' $g$ in $\mathbb{F}_X$.

\begin{definition}
Recall that a subset $C$ of a poset $(P, \le)$ is said to be {\em cofinal\/} in $(P, \le)$ provided that for every 
$p\in P$ there exists $q\in C$ such that $p\le q$.
\end{definition}

\begin{definition}
We say that a family $\mathscr{F}\subseteq \mathbb{F}_X$ is {\em countably closed\/} if $\triangle \mathscr{F}'\in \mathscr{F}$ for every countable subfamily  $\mathscr{F}'$ of $\mathscr{F}$.
\end{definition}

The following theorem is the main technical tool in this paper.

\begin{theorem}
\label{reflecting:pseudocompletness}
Let $X$ be a topological space,  
$\mathscr{F}\subseteq \mathbb{F}_X$
be a countably closed family and
$\{{\cal B}_n: n\in\N\}$ be a countable sequence
such that
\begin{equation}
\label{inside:F}
\bigcup\{{\cal B}_n: n\in\N\}\subseteq \left(\bigcup_{f\in\mathscr{F}} {\cal P}_f\right) \cap
\left(\bigcup_{f\in\mathscr{F}} {\cal I}_f\right).
\end{equation}
For every $f\in\mathscr{F}$ and each $n\in\N$, define 
\begin{equation}
\label{eq:C_n,f}
{\cal C}_{n,f}=\{f(B):B\in{\cal B}_n\cap{\cal P}_f\cap
 {\cal I}_f\}.
\end{equation}
\begin{itemize}
\item[(a)] If each ${\cal B}_n$ is a pseudobase of $X$, then 
the set 
$$
\mathscr{F}_1=\{f\in\mathscr{F}: {\cal C}_{n,f}
\mbox{ is a pseudobase for }
X 
\mbox{ for each }n\in \N\}
$$ 
is cofinal in $(\mathscr{F},\preceq)$.
\item[(b)] If $\{{\cal B}_n: n\in\N\}$ is a complete sequence in $X$, then 
the set 
$$
\mathscr{F}_2=\{f\in\mathscr{F}: \{{\cal C}_{n,f}: n\in \N\}
\mbox{ is a complete sequence in }
f(X)\}
$$
 is cofinal in $(\mathscr{F},\preceq)$; in particular, the set 
$$
\mathscr{F}_3=\{f\in\mathscr{F}: f(X) \mbox{ is \Tp}\}
$$ is cofinal in $(\mathscr{F},\preceq)$.
\end{itemize}
\end{theorem}

\proof
(a)
We shall first prove the following
\begin{claim}
\label{claim:1}
For every $f\in\mathscr{F}$ there exists $g\in\mathscr{F}$ such that:
\begin{itemize} 
\item[(i)]
$f\preceq g$;
\item[(ii)]
for every $n\in\N$ and each non-empty open subset $V$ of $f(X)$,
there exists $B_{f,V}^n\in {\cal B}_n\cap{\cal P}_{g} \cap{\cal I}_{g}$ 
satisfying  
$B_{f,V}^n\subseteq f^{\leftarrow}(V)$.
\end{itemize}
\end{claim}

\proof
Fix 
$f\in\mathscr{F}$. 
Let  
$\mathcal{U}$ be a countable base of $f(X)$.
Let 
$U\in \mathcal{U}\setminus\{\emptyset\}$.
Then $f^{\leftarrow}(U)$ is a  non-empty open subset of $X$.
Let $n\in\N$ be arbitrary.
Since ${\cal B}_n$ is a pseudobase in $X$,
there exists 
$C_U^n\in\cal B_n$ such that $C_U^n\subseteq f ^{\leftarrow }(U)$. 
\eqref{inside:F} allows us to find
maps $f_U^n,g_U^n\in \mathscr{F}$
such that $C_U^n\in{\cal P}_{f_U^n} \cap{\cal I}_{g_U^n}$. 
Define 
$$
g=\bigtriangleup \{f_U^n\triangle g_U^n: U\in \mathcal{U},n\in\N\}
\triangle f.  
$$

Since $\mathscr{F}$ is countably closed, $g\in \mathscr{F}$.
Clearly, $f \preceq g$, so (i) is satisfied.

Let us check (ii). Let $n\in\N$ and let $V$ be a non-empty open subset of $f(X)$.
Since $\mathcal{U}$ is a base of $f(X)$, we can  find $U\in\mathcal{U}\setminus\{\emptyset\}$ such that $U\subseteq V$.
We claim that $B^n_{f,V}=C^n_{U}$ is  as required.
Indeed, since $C^n_{U}\in {\cal P}_{f_U^n}$ and 
$f_U^n\preceq g$, Lemma \ref{upwards:directed}(i) implies
$C^n_{U}\in {\cal P}_g$.
Similarly, since $C^n_{U}\in {\cal I}_{g_U^n}$ and 
$g_U^n\preceq g$, Lemma \ref{upwards:directed}(ii) implies
$C^n_{U}\in {\cal I}_g$.
Since $C^n_{U}\in{\cal B}_n$ by our choice,
we conclude that $B^n_{f,V}=C^n_{U}\in {\cal B}_n\cap{\cal P}_g \cap{\cal I}_g$.
Finally, $B^n_{f,V}=C^n_{U}\subseteq  f^{\leftarrow }(U)\subseteq f^{\leftarrow }(V)$ by our choice of $C^n_{U}$ and the inclusion $U\subseteq V$.
\endproof

Let $h\in\mathscr{F}$ is arbitrary. To prove that $\mathscr{F}_1$ is cofinal
in $(\mathscr{F},\preceq)$, we shall find $f\in \mathscr{F}_1$ with $h\preceq f$.
Starting from $h_0=h$, we use induction and Claim \ref{claim:1} to define a
sequence $\{h_i:i\in\N\}\subseteq \mathscr{F}$ satisfying the following conditions:
\begin{itemize}
\item[($\alpha$)] $h_0 \preceq h_1\preceq h_2\preceq\ldots\preceq h_i\preceq h_{i+1}\preceq\ldots$;
\item[($\beta$)] if $j,n\in\N$ and $V$ is a non-empty open subset of $h_j(X)$,
then $B_{h_j,V}^n\subseteq h_{j}^{\leftarrow }(V)$
for some $B_{h_j,V}^n\in {\cal B}_n\cap {\cal P}_{h_{j+1}} \cap{\cal I}_{h_{j+1}}$.
\end{itemize} 

Since $\mathscr{F}$ is countably closed, 
\begin{equation}
f=\triangle _{j\in \N }h_j\in\mathscr{F}.
\end{equation}
 Since $h=h_0\preceq f$, it remains only to show that $f\in\mathscr{F}_1$.

Fix $n\in\N$. We need to show that ${\cal C}_{n,f}$ is a pseudobase of $f(X)$. 
To achieve this, we shall check conditions (i) and (ii) of Definition \ref{def:pi-pseudobase}.

(i) Let $C\in {\cal C}_{n,f}$. By \eqref{eq:C_n,f}, $C=f(B)$ for some 
$B\in{\cal B}_n\cap{\cal P}_f\cap {\cal I}_f$.
Since $B\in {\cal I}_f$, Notation \ref{non-empty-interior:sets} implies that
$\Int_{f(X)}(C)=\Int_{f(X)}(f(B))\not=\emptyset$.
 
(ii)
 Take an arbitrary non-empty open set $W$ in $f(X)$.
Fix $x_0\in X$ such that $f(x_0)\in W$.
There exist a natural number $j\in \N $ and  open sets $U_i\subseteq h_i(X)$ for 
every $i\leq j$ such that 
\begin{equation}
\label{eq1}
(\Pi _{i\leq j}U_i)\times (\Pi _{i>j}h_i(X))\cap f (X)\subseteq W
\end{equation}
and
\begin{equation}
\label{eq:8}
h_i(x_0)\in U_i
\mbox{ for all }
i\le j.
\end{equation}

For each $i\le j$, $h_i\preceq h_j$ holds by ($\alpha$), so there exists a continuous map 
$q^j_i: h_j(X)\to h_i(X)$ such that $h_i=q^j_i\circ h_j$.
Combining this with \eqref{eq:8}, we conclude that
$q^j_i(h_j(x_0))=q^j_i\circ h_j(x_0)=h_i(x_0)\in U_i$ for every $i\le j$,
so 
\begin{equation}
\label{eq:proj}
q(h_j(x_0))\in \prod_{i\le j} U_i=U,
\end{equation}
where 
$$
q=\triangle_{i\leq j}q^j_i: h_j(X)\to\prod_{i\le j} h_i(X)
$$
is the diagonal map.
Since each $q^j_i$ is continuous, so is $q$.
Since $U=\Pi _{i\leq j} U_i$ is an open subset of $\prod_{i\le j} h_i(X)$,
\begin{equation}
\label{def:V}
V=q^{\leftarrow } (U)
\end{equation}
is an open subset of $h_j(X)$.
From \eqref{eq:proj}, we get
$h_j(x_0)\in q^{\leftarrow}(U)=V$.
This shows that $V\not=\emptyset$.
Applying condition ($\beta$) to this $V$, we can find $B_{h_j,V}^n\in {\cal B}_n\cap {\cal P}_{h_{j+1}} \cap{\cal I}_{h_{j+1}}$
satisfying 
\begin{equation}
\label{B:in:V}
B_{h_j,V}^n\subseteq h_{j}^{\leftarrow }(V).
\end{equation}
From $h_{j+1}\preceq f$ and Lemma \ref{upwards:directed}
we conclude that
${\cal P}_{h_{j+1}}\subseteq {\cal P}_f$ and 
${\cal I}_{h_{j+1}}\subseteq {\cal I}_f$.
Therefore, $B_{h_j,V}^n\in{\cal B}_n\cap {\cal P}_f\cap {\cal I}_f$,
which implies $f(B_{h_j,V}^n)\in {\cal C}_{n,f}$ by \eqref{eq:C_n,f}.
So it remains only to show that $f(B_{h_j,V}^n)\subseteq W$.

Let $\pi_j:\prod_{i\in \N} h_i(X)\to h_j(X)$ be the projection on the $j$th coordinate, and let $p_j=\pi_j\restriction_{f(X)}:f(X)\to h_j(X)$ be its restriction to $f(X)$.
Since 
$$
q\circ p_j\circ f=q\circ h_j=\triangle_{i\leq j}q^j_i\circ h_j=\triangle_{i\leq j} h_i,
$$
it follows from \eqref{eq1} and \eqref{def:V} that
\begin{equation}
p_j^{\leftarrow}(V)=p_j^{\leftarrow}(q^{\leftarrow}(U))=(q\circ p_j)^{\leftarrow}(U)\subseteq W.
\end{equation}
Therefore,
$p_j^{\leftarrow}(V)=p_j^{\leftarrow}(q^{\leftarrow}(U))\subseteq W$.
Since $h_j=p_j\circ f$, from this and \eqref{B:in:V} we get
$$
B_{h_j,V}^n\subseteq h_{j}^{\leftarrow }(V)
=
(p_j\circ f)^{\leftarrow}(V)
=
f^{\leftarrow}(p_j^{\leftarrow }(V))
\subseteq f^{\leftarrow}(W),
$$
which implies
$f(B_{h_j,V}^n)\subseteq f(f^{\leftarrow}(W))\subseteq W$.

\medskip
(b)
Since the sequence $\{{\cal B}_n: n\in\N\}$ is complete, Definition \ref{def:nested:and:complete:sequences}(ii) 
implies
that each ${\cal B}_n$ is a pseudobase for $X$; that is, the assumption of item (a) 
holds. 
Applying the conclusion of 
item (a), we conclude that $\mathscr{F}_1$ is cofinal in $(\mathscr{F},\preceq)$.
Therefore, it suffices to check that $\mathscr{F}_1\subseteq \mathscr{F}_2$.

Let $f\in \mathscr{F}_1$ be arbitrary. It follows from the definition of $ \mathscr{F}_1$
that
$\mathcal{C}_{n,f}$ is a pseudobase
of $f(X)$ 
for every $n\in \N$.
Since the sequence $\{{\cal B}_n: n\in\N\}$ is complete, it is centered by Definition \ref{def:nested:and:complete:sequences}(ii).
It follows from Remark \ref{sufamily:of:a:centered:family:is:centered}(i) that the sequence $\{{\cal B}_n\cap{\cal P}_f\cap {\cal I}_f: n\in\N\}$ is centered as well.
Applying Lemma \ref{image:of:a:centered:sequence}(iii), we conclude that 
the sequence  $\{{\cal C}_{n,f}: n\in \N\}$ is centered as well.
According to Definition \ref{def:nested:and:complete:sequences}(ii), 
$\{{\cal C}_{n,f}: n\in \N\}$ is a complete sequence in $f(X)$.
Therefore $f\in \mathscr{F}_2$ by the definition of $\mathscr{F}_2$.
\endproof

\section{Polish factorizable groups: Proof of Theorem \ref{equivalent:conditions:for:precompact}}
\label{sec:polish;factorizable}

In this section we study properties of Polish factorizable groups. The reader is referred to Definition \ref{def:polish:factorizable} for their definition.

\begin{notation}
When $X$ is a topological group, we consider a subset 
$$
\mathbb{H}_X=\{f\in \mathbb{F}_X: f(X) \mbox{ is a group and } f:X\to f(X)
\mbox{ is a group homomorphism}\} 
$$
of $\mathbb{F}_X$.
\end{notation}

\begin{definition}
\label{def:R-factorizable}
When $\mathbb{H}_X$ is cofinal in $(\mathbb{F}_X, \preceq)$, a topological group $X$ is called {\em $\bb R$-factorizable\/}.
\end{definition}

The original definition of an $\mathbb{R}$-factorizable group due to M.G. Tka\v{c}enko \cite{Tkachenko} was different but it is known to be equivalent to the one above; see \cite{AT}.

We shall sat that a topological group $G$ is {\em Polishable\/} if $G$ is separable metric and its topology can be generated by some complete metric on $G$, or equivalently, if $G$ is a \v{C}ech-complete separable metric topological group. 

\begin{theorem}
\label{SOp:are*polish:factorizable}
Assume that  
$X$ is a topological  group
satisfying one of the following conditions:
\begin{itemize}
\item[(a)] $X$ is an $\bb R$-factorizable \SOp \ topological group;
\item[(b)] $X$ is a 
c.c.c. \Op\ topological group.
\end{itemize}
If
$\mathscr{F}$ is a countably closed cofinal subfamily of 
$(\mathbb{H}_X,\preceq)$,
then the set 
$$
\mathscr{P}({\mathscr{F}})=\{f\in\mathscr{F}: f(X)
\mbox{ is a Polishable group}
\}
$$ 
is cofinal in 
$(\mathscr{F},\preceq)$ and thus, also in 
$(\mathbb{H}_X,\preceq)$.
\end{theorem}

\proof
Our goal is to find a complete sequence $\{\B_n:n\in\N\}$ in $X$ 
satisfying the condition \eqref{inside:F} from the assumption of Theorem \ref{reflecting:pseudocompletness}.
The definition of this sequence would be different in cases (a) and (b), so we consider two cases.

\medskip
{\em Case (a)\/}. First note that
$\mathscr{F}$ is cofinal in $(\mathbb{F}_X,\preceq)$.
Indeed, $\mathscr{F}$ is cofinal in $(\mathbb{H}_X,\preceq)$
by our assumption, and the latter set is cofinal in $(\mathbb{F}_X,\preceq)$ as $X$ is $\R$-factorizable; see Definition \ref{def:R-factorizable}.

Since $X$ is \SOp, 
by Definition \ref{four:properties}(iii), $X$ has a complete sequence $\{\cal{B}_n: n\in\N\}$
such that each 
$\cal{B}_n$ consists of zero-sets in $X$.

Fix $B\in\bigcup\{{\cal B}_n:n\in\N\}$. Then $B\in {\cal B}_n$ for some $n\in\N$. Since $B$ is a zero-set in $X$, we can choose a continuous function
$f:X\to \bb{R}$ such that $B=f^{-1}(0)$. Then $f\in\mathbb{F}_X$ and $B\in{\cal P}_f$. Since 
$\mathscr{F}$ is cofinal in $(\mathbb{F}_X,\preceq)$,
there exists $g\in\mathscr{F}$ such that
$f\preceq g$.
Thus, $B\in {\cal P}_f\subseteq {\cal P}_g$ by Lemma \ref{upwards:directed}(i).
This shows that $B\in\bigcup_{f\in\mathscr{F}} {\cal P}_f$.

Since $B\in{\cal B}_n$ and ${\cal B}_n$ is a pseudobase,
$\Int_X(B)\not=\emptyset$ by Definition \ref{def:pi-pseudobase}(i).
Fix $x_0\in \Int_X(B)$.
Since $X$ is Tychonoff, there exists a continuous function
$f:X\to[0,1]$ such that 
$f(x_0)=1$ and $f(X\setminus \Int_X(B))\subseteq \{0\}$.
Then $F=f^{\leftarrow}((0,1]\cap f(X))\in{\cal P}_f$,
$x_0\in \Int_{f(X)}(f(F))=(0,1]\cap f(X)\not=\emptyset$
and $F\subseteq \Int_X(B)\subseteq B$.
This means that $B\in{\cal I}_f$ by Notation \ref{non-empty-interior:sets}.
Since 
$\mathscr{F}$ is cofinal in $(\mathbb{F}_X,\preceq)$,
there exists $g\in\mathscr{F}$ 
such that 
$f\preceq g$.
So $B\in {\cal I}_f\subseteq {\cal I}_g$ by Lemma \ref{upwards:directed}(ii).
This shows that $B\in\bigcup_{f\in\mathscr{F}} {\cal I}_f$.

\medskip
{\em Case (b)\/}. Since $X$ is c.c.c., it is $\omega$-bounded in Guran's sense
\cite{Guran}; that is, there exists a family $\{H_i:i\in I\}$ of separable metric groups $H_i$ so that $X$ is (isomorphic to) a subgroup of
the product $H=\{H_i:i\in I\}$.
Therefore, the family ${\cal V}=\{f^{\leftarrow}(O): f\in\mathbb{H}_X, O \text{ is a non-empty open subset of } f(X)\}$
is a base of $X$.

\begin{claim}
\label{dense:open:claim}
Every non-empty open subset $W$ of $X$ contains a dense open subset $U_W$ such that $U_W\in {\cal P}_f\cap {\cal I}_f$ for some $f\in\mathscr{F}$. 
\end{claim}
\begin{proof}
For every open subset $W$ of $X$, let ${\cal V}_W$ be a maximal subfamily of ${\cal V}$ consisting of pairwise disjoint subsets of $W$; such a subfamily exists by Zorn's lemma.
Since $X$ is c.c.c., ${\cal V}_W$ is at most countable, so we can write ${\cal V}_W=\{V_n:n\in\N\}$, where 
$V_n=f_n^{\leftarrow}(O_n)$ for a suitable $f_n\in\mathbb{H}_X$ and 
 a non-empty open subset $O_n$ of $f_n(X)$ for every $n\in\N$.
Furthermore, $U_W=\bigcup {\cal V}_W=\bigcup_{n\in \N} V_n$ is 
dense in $W$ by the maximality of ${\cal V}_W$ and the fact that ${\cal V}$ is a base of $X$. Since $f_n$ is continuous and $O_n$ is open in $f_n(X)$, the set $V_n$ is open in $X$ for every $n\in\N$.
Thus, $U_W$ is open in $X$ as well. 

Note that $g=\triangle_{n\in\N} f_n\in\mathbb{H}_X$.
Since $\mathscr{F}$ is cofinal in
$(\mathbb{H}_X,\preceq)$, we can fix $f\in \mathscr{F}$
with $g\preceq f$.

Let $n\in\N$. Since $V_n=f_n^{\leftarrow}(O_n)$, we have 
$V_n\in{\cal P}_{f_n}$. Since $f(V_n)=O_n$ is a non-empty open subset
of $f_n(X)$,
$V_n\in{\cal I}_{f_n}$ by Notation \ref{non-empty-interior:sets}.
Since $f_n\preceq f$, from Lemma \ref{upwards:directed}
we get $V_n\in{\cal P}_f\cap{\cal I}_f$.

One can easily see that the family ${\cal P}_f$ is closed under taking arbitrary unions. Since $V_n\in{\cal P}_f$ for all $n\in\N$,
we obtain $U_W=\bigcup_{n\in\N} V_n\in{\cal P}_f$.
Since $V_1\in{\cal I}_f$ and $V_1\subseteq U_W$, it follows from 
Notation \ref{non-empty-interior:sets} that $U_W\in{\cal I}_f$.
\end{proof}

Since $X$ is \Op, by Definition \ref{four:properties}(ii), $X$ 
has a complete sequence $\{\A_n:n\in\N\}$ consisting of open subsets of $X$.
For every $n\in\N$ and each $A\in\A_n$, use Claim 
\ref{dense:open:claim} to fix a dense open subset $U_{n,A}$ of $A$
such that 
$U_{n,A}\in {\cal P}_f\cap {\cal I}_f$ for some $f\in\mathscr{F}$. 
Define $\B_n=\{U_{n,A}:A\in\A_n\}$.
Then $\{\B_n:n\in\N\}$ is a complete sequence in $X$ by 
Corollary \ref{dense:sequences}.
Furthermore, \eqref{inside:F} holds by our choice of $U_{n,A}$'s.

\medskip
Having constructed a complete sequence $\{\B_n:n\in\N\}$ satisfying \eqref{inside:F} separately in each of the two cases, we now return back to a common proof.
By Theorem 
\ref{reflecting:pseudocompletness}(b), 
$\mathscr{F}_3=\{f\in\mathscr{F}: f(X)$ is \Tp$\}$ 
is cofinal in $(\mathscr{F},\preceq)$.

Let $f\in\mathscr{F}_3$.
By Theorem \ref{metric:Todd:are:Cech-complete},
$f(X)$ is a \v{C}ech-complete separable metric group, so 
$f(X)$ is Polishable. This shows that
$\mathscr{F}_3\subseteq \mathscr{P}(\mathscr{F})$.
Since $\mathscr{F}_3$ is cofinal in $(\mathscr{F},\preceq)$, so is 
$\mathscr{P}(\mathscr{F})$.
\endproof

\begin{corollary}
\label{polish_fact}
A topological group $X$ is Polish factorizable whenever $G$ satisfies one of the following two conditions:
\begin{itemize}
\item[(a)] $X$ is an $\bb R$-factorizable \SOp \ topological group;
\item[(b)] $X$ is 
a 
c.c.c. \Op\ topological group.
\end{itemize}
\end{corollary}
\begin{proof}
The family $\mathbb{H}_X$ is countably closed, so we can apply Theorem \ref{SOp:are*polish:factorizable} with $\mathscr{F}=\mathbb{H_X}$ to conclude that the set $\mathscr{P}(\mathscr{F})$ from the conclusion of this theorem is cofinal in $(\mathbb{H}_X,\preceq)$. 
\end{proof}

\begin{proposition}
\label{SOp:precompact:are:psc}
Every precompact Polish factorizable group is pseudocompact.
\end{proposition}
\proof
Let $G$ be a 
precompact 
Polish factorizable
group. 
Being precompact, $G$ is $\bb R$-factorizable.

Let $\varphi:G\to \bb R$ be a continuous function. 
Since $G$ is $\bb R$-factorizable, there exists a topological group $K$ with a countable base, a continuous surjective homomorphism $f :G\to K$ and a continuous function $g:K\to \bb R$ such that $\varphi=g\circ f $. 
Since $G$ is Polish factorizable,
there exist a Polish group $H$, a continuous surjective homomorphism $\pi :G\to H$ 
and a continuous homomorphism  $h:H\to K$ such that $f=h\circ \pi $. 
Since $G$ is precompact and $\pi$ is surjective, the group $H=\pi(G)$ is precompact, and so the completion $\varrho H$ is compact. On the other hand,
$H$ is a Polish group, so $H$ must be closed in $\varrho H$ by Corollary \ref{dense_cech}. This means that $H=\varrho H$ is compact, so
its image $g\circ h(H)$ is a compact (thus bounded) subset of $\bb R$.
Since $\varphi=g\circ h\circ \pi$, it follows that the set $\varphi(G)=g\circ h(H)$
is bounded in $\bb R$.
\endproof

\medskip
\noindent
{\bf Proof of Theorem \ref{equivalent:conditions:for:precompact}:}
(i)$\Rightarrow$(ii) is clear. 

(ii)$\Rightarrow$(ix) follows from Proposition \ref{weakly:psc:are:SO}(ii).

(i)$\Rightarrow$(viii) follows from Proposition \ref{loc:psc:is:Telgarsky}.

(iii)$\Rightarrow$(x) and (iv)$\Rightarrow$(x). 
Since $G$ is precompact, it is $\R$-factorizable \cite{Tkachenko}
and c.c.c.
Applying Corollary \ref{polish_fact}, 
we conclude that $G$ is Polish factorizable. 

(x)$\Rightarrow$(i).  
By Proposition \ref{SOp:precompact:are:psc}, $G$ is pseudocompact.

The rest of implications follow from Diagram 2.
\qed

\section{Completeness in products: A proof of Theorem \ref{equivalence:for:dense:subgroups:without:cech} }
\label{sec:proofs}

\begin{theorem}
\label{O-omega-discrete}
\label{necessary:condition:for:completeness:of:a:dense:subgroup}
Let $H=\prod_{i\in I} H_i$ be  a product of groups $H_i$ with a countable base.  For every $J\subseteq I$, let
$\pi_J: H\to H_J=\prod_{i\in J} H_i$ be the corresponding projection.
Suppose that 
$X$ is a dense subgroup of $H$ such that   
$X$ is either \SOp\/ or \Op. Then:
\begin{itemize}
\item[(a)] $\pi_J(X)=H_J$ for every countable subset 
$J$ of $I$;
\item[(b)]  $H_i$ is  \v{C}ech-complete for all $i\in I$.
\end{itemize}
\end{theorem}

\proof 
Without loss of generality, we may assume that all $H_i$ are non-trivial.

First, we are going to prove the following
\begin{claim}
\label{Raikov:complete:case}
The statement of item (a) holds in the special case when all $H_i$
are Raikov complete.  
\end{claim}
\begin{proof}
Suppose that $H_i$ is  
Raikov complete
for every $i\in I$.
For every $J\subseteq I$, let $p_J=\pi_J\restriction_X$.
We 
claim
that the family
$$
\mathscr{F}=\{p_J: J \mbox{ is a countable subset of }I\}
\subseteq \mathbb{H}_X
$$
satisfies all the hypothesis of Theorem
\ref{SOp:are*polish:factorizable}. Clearly, $\mathscr{F}$ is countably closed.
To show that $\mathscr{F}$ is cofinal in $(\mathbb{H}_X,\preceq)$, 
take $g\in \mathbb{H}_X$.
By \cite[Corollary 1.7.8]{AT}, there exist a countable set $J\subseteq I$ and a continuous mapping $h:\pi _J(G)\to g(X)$ such that 
$g=h\circ p_J$.
Thus, $p_J\in \mathscr{F}$ and $g\preceq p_J$.
 
Being a dense subgroup of $H$, $X$ is $\R$-factorizable
by \cite[Corollary 8.1.15]{AT}. 
Since $H$ is c.c.c, so is $X$.
Finally, if
$X$ is either \SOp\/ or \Op\/, the hypothesis of Theorem \ref{SOp:are*polish:factorizable} are satisfied. Applying this theorem, we conclude that
the set $\mathscr{P}(\mathscr{F})
=\{f\in\mathscr{F}: f(X)$ is a Polishable group$\}$ 
is cofinal in $(\mathscr{F},\preceq)$. 

Let $J$ be a countable subset of $I$. 
Since $p_J\in\mathscr{F}$ and $\mathscr{P}(\mathscr{F})$ is 
cofinal in $(\mathscr{F},\preceq)$, there exists a countable subset $M\subseteq I$ such that $p_J\preceq p_M$ and $p_M(X)$ is a Polishable group. 
Since $p_J\preceq p_M$, there exists a continuous function
$h:p_M(X)\to p_J(X)$ satisfying $p_J=h\circ p_M$.
Since $p_J$ is a homomorphism, so is $h$. 
Let $\tilde{h}:\varrho(p_M(X))\to \varrho(p_J(X))$ be the continuous homomorphism extending $h$ over the completions.

Since both $\pi_J$ and $\pi_M$ are continuous and $X$ is a dense subgroup of $H$, $p_J(X)=\pi_J(X)$ is a dense subgroup of $\pi_J(H)=H_J$
and 
$p_M(X)=\pi_M(X)$ is a dense subgroup of $\pi_M(H)=H_M$.
Since both $H_J$ and $H_M$ are 
Raikov complete
groups,
$\varrho(p_J(X))=H_J$ and $\varrho(p_M(X))=H_M$.
This shows that $\tilde{h}: H_M\to H_J$.

Note that $\tilde{h}\circ \pi_M\restriction_X=
\tilde{h}\circ p_M=h\circ p_M=p_J=\pi_J\restriction_X$.
Since $X$ is dense in $H$, this means that $\tilde{h}\circ \pi_M=\pi_J$.

We claim that $J\subseteq M$. Indeed, suppose the contrary.
Then we can find $j\in J\setminus M$.
Since $H_j$ is non-trivial, there exists $y_j\in H_j$ with $y_j\not= e_j$, where $e_j$ is the identity element of $H_j$.
Let $g=(g_i)_{i\in I}\in H$ be an element defined 
by $g_j=x_j$ and $g_i=e_i$ for all $i\in I$ with $i\not= j$, where 
$e_i$ is the identity element of $H_i$.
Then
$\pi_J(g)\not=e_J$, as $g_j=x_j\not=e_j$ for  $j\in J$.
On the other hand, since $j\not\in M$, we have $g_i=e_i$ for all
$i\in M$, which implies $\pi(g)=e_M$
and $\tilde{h}\circ \pi(g)=\tilde{h}(e_M)=e_J$, as $\tilde{h}$ is a homomorphism. We have shown that
$\pi_J(g)\not=e_J=\tilde{h}\circ \pi(g)$; that is, $\pi_J\not=\tilde{h}\circ \pi_M$. This contradiction shows that $J\subseteq M$.

Let $\pi^M_J:H_M\to H_J$ be the projection corresponding to the inclusion $J\subseteq M$.

Since $X$ is dense in $H$ and $\pi_M$ is continuous, $\pi_M(X)$ is dense in $H_M$.
Since $\pi_M$ is a homomorphism, $\pi_M(X)$ is a subgroup of $H_M$. Since $\pi_M(X)=p_M(X)$ is \v{C}ech-complete, it follows from
Corollary \ref{dense_cech} that $\pi_M(X)=H_M$.
Therefore, $p_J(X)=\pi_J(X)=\pi^M_J\circ\pi_M(X)=\pi^M_J(\pi_M(X))
=\pi^M_J(H_M)=H_J$.
\end{proof}

Returning back to the proof of the theorem, we consider the general case when $H_i$ are separable metric groups.
For each $i\in I$, $H_i$ is dense in its Raikov completion $\varrho(H_i)$. Therefore, $H=\prod_{i\in I}  H_i$ is dense in
$\tilde{H}=\prod_{i\in I} \varrho(H_i)$. Since $X$ is dense in $H$ by our 
assumption, $X$ is a dense subgroup of $\tilde{H}$.
Applying Claim \ref{Raikov:complete:case} to the dense subgroup 
$X$ of the product $\tilde{H}=\prod_{i\in I} \varrho(H_i)$, 
we conclude that
$p_J(X)=\prod_{i\in J} \varrho(H_i)$ for every countable subset $J$ of $I$.
Since $X\subseteq \prod_{i\in I} H_i$, it follows that
$\prod_{i\in J} H_i\subseteq \prod_{i\in J} \varrho(H_i)=p_J(X)\subseteq \prod_{i\in J} H_i$.
We have checked that
$p_J(X)=\prod_{i\in J} H_i=\prod_{i\in J} \varrho(H_i)$
for every countable set $J\subseteq I$.
In particular, for $J=\{i\}$, we conclude that $H_i=\varrho(H_i)$;
that is, $H_i$ is Raikov complete.
It remain only to note that Raikov complete separable metric groups are \v{C}ech-complete;
see \cite[Proposition 4.3.8]{AT}.
\endproof

\begin{notation}
Let $X=\prod_{i\in I} X_i $ be a product of non-empty topological spaces $X_i$.
For each non-empty set $J\subseteq I$, we denote by $\pi_J$
the projection of $X$ into its subproduct $X_J=\prod_{i\in J} X_i$.
\end{notation}

\begin{definition}
A subset $B$ of the product $\prod_{i\in I} X_i$ shall be called {\em basic\/}
if $B=\prod_{i\in I} B_i$, where $B_i\subseteq X_i$ for all $i\in I$ and the set 
$\mathrm{supp}(B)=\{i\in I: B_i\not= X_i\}$ is finite.
\end{definition}

\begin{remark}
\label{remarks:closures:and:interiors}
If $B=\prod_{i\in I} B_i$ is a basic subset of $X=\prod_{i\in I} X_i$,
then both
$$
\Cl_{X}(B)=\prod_{i\in I} \Cl_{X_i} (B_i)
\
\mbox{ and }
\
\Int_{X}(B)=\prod_{i\in I} \Int_{X_i} (B_i)
$$
are basic subsets of $X$ such that
$\mathrm{supp}(\Cl_X(B))\subseteq \mathrm{supp}(B)$
and 
$\mathrm{supp}(\Int_X(B))= \mathrm{supp}(B)$.
\end{remark}

\begin{lemma}
\label{basic:sets}
Let $X=\prod_{i\in I} X_i $ be a product of non-empty topological spaces $X_i$,
and let $Y$ be a subspace of $X$ such that $\pi_J(Y)=X_J$ for every finite set $J\subseteq I$. Then:
\begin{itemize}
\item[(i)] $\Cl_Y(B\cap Y)=\Cl_X(B)\cap Y$ for every basic subset $B$ of $X$;
\item[(ii)] if $A$ and $B$ are basic subsets of $X$ such that $A\cap Y\subseteq B\cap Y$, then $A\subseteq B$. 
\end{itemize}
\end{lemma}
\begin{proof}
(i) The inclusion $\Cl_Y(B\cap Y)\subseteq\Cl_X(B)\cap Y$ is clear.
Let us check the inclusion $\Cl_X(B)\cap Y\subseteq \Cl_Y(B\cap Y)$.
Let $y\in \Cl_X(B)\cap Y$ be arbitrary.
Let $V$ be a basic open neighbourhood of $y$ in $X$.
Since $y\in \Cl_X(B)$, we have $B\cap V\not=\emptyset$,
so $x\in B\cap V$ for some $x=(x_i)_{i\in I}\in X$.
Since $B$ and $V$ are basic sets, $J=\mathrm{supp}(B)\cup \mathrm{supp}(V)$ is finite. Since $\pi_J(Y)=X_J$ by our assumption, 
there exists $z=(z_i)_{i\in I}\in Y$ such that $z_i=x_i$ for all $i\in J$.
Now $x\in B\cap V$ implies $z\in B\cap V$.
Thus, $z\in (B\cap Y)\cap V\not=\emptyset$. 
We showed that $(B\cap Y)\cap V\not=\emptyset$ for every basic open neighbourhood $V$ of $y$ in $X$. This implies $y\in \Cl_Y(B\cap Y)$.

(ii) Since $A$ and $B$ are basic, the set $J=\mathrm{supp}(A)\cup \mathrm{supp}(B)$ is finite. Let $a=(a_i)_{i\in I}\in A$ be arbitrary.
Since $\pi_J(Y)=X_J$ by our assumption, 
there exists $y=(y_i)_{i\in I}\in Y$ such that $y_i=a_i$ for all $i\in J$.
Note that $y\in A$, as $a\in A$ and $\mathrm{supp}(A)\subseteq J$.
Therefore, $y\in A\cap Y$. Since $A\cap Y\subseteq B\cap Y$ by the assumption of (ii), it follows that 
$y\in B$. Since $\mathrm{supp}(B)\subseteq J$ and $y_i=a_i$ for all $i\in J$, 
we conclude that $a\in B$.
\end{proof}

\begin{lemma}\label{product_sequences}
Let $X=\prod_{i\in I} X_i $ be a product of non-empty topological spaces $X_i$.
For every $n\in \N$ and each $i\in I$, let 
$\{\cal B_{i,n}:n\in\N\}$ be a centered sequence
in $X_i$ such that $X_i\in \cal B_{i,n}$.
For every $n\in\N$, 
define
\begin{equation}
\label{rectangular:products}
\B_n=\{B\subseteq X: B=\prod_{i\in I} B_i
\mbox{ is a basic subset of } X
\mbox{ such that }
B_i\in \cal B_{i,n}
\mbox{ for all }
i\in I\}.
\end{equation}
Let $Y$ be a subspace of $X$ such that
$\pi _J(Y)=X_J$ 
for every at most countable subset $J$ of $I$.
Define
\begin{equation}
\label{eq:C_n}
\cal C_n=\{B\cap Y:B\in\cal B_n\}
\ 
\mbox{ for every }
n\in\N.
\end{equation}
Then
$\{\cal C_n:n\in \N\}$ is a centered sequence in $Y$.
\end{lemma}

\proof
We first prove the following 
\begin{claim}
\label{centered:claim}
The sequence $\{\B_n:n\in\N\}$ is centered in $X$.
\end{claim}
\begin{proof}
Let $\{B_n:n\in\N\}$ be a nested sequence in $X$ such that $B_n\in\cal B_n$ for each $n\in \N$.  
For $n\in\N$, we can use \eqref{rectangular:products}
to fix 
$B_{i,n}\in \B_{i,n}$ for every $i\in I$ such that
\begin{equation}
\label{eq:B_n:product}
B_n=\prod_{i\in I} B_{i,n}.
\end{equation}

Let $n\in\N$.
Since the sequence $\{B_n:n\in\N\}$ is nested,
$\cl _X(B_{n+1})\subseteq \Int_X(B_n)$.
Since
$$
\Cl_X(B_n)=\prod_{i\in I} \Cl_{X_i}(B_{i,n})
\ 
\mbox{ and }
\ 
\Int_X(B_n)=\prod_{i\in I} \Int_{X_i}(B_{i,n}),
$$
we conclude that
$\Cl_{X_i}(B_{i,n+1})\subseteq \Int_{X_i}(B_{i,n})$
for every $i\in I$. 

Let $i\in I$ be arbitrary. The above argument shows that the sequence
$\{B_{i,n}:n\in\N\}$ is nested. Since $B_{i,n}\in \B_{i,n}$ for every $n\in\N$
and the sequence $\{\B_{i,n}:n\in\N\}$ is centered in $X_i$, we can choose
\begin{equation}
\label{point]in]intersection}
x_i\in\bigcap_{n\in \N} B_{i,n}.
\end{equation} 
From \eqref{eq:B_n:product} and \eqref{point]in]intersection}, we conclude 
that 
$x=(x_i)_{i\in I}\in \bigcap_{n\in\N}B_n\not=\emptyset$. 
This proves that $\{\cal B_n:n\in\N\}$ is a 
centered sequence in $X$.
\end{proof}

We are going to show that
$\{\cal C_n:n\in\N\}$ is a centered sequence in $Y$, where 
$\cal C_n$ are as defined in
\eqref{eq:C_n}.
Let $\{C_n:n\in\N\}$ be a nested sequence in $Y$ such that 
$C_n\in{\cal C}_n$ for every $n\in\N$.
Use \eqref{eq:C_n} to fix $B_n\in\B_n$ such that 
$B_n\cap Y=C_n$ for every $n\in\N$.
Then 
\begin{equation}
\label{nested:in:Y}
\Cl _Y(B_{n+1}\cap Y)=\Cl_Y(C_{n+1})\subseteq \Int_Y(C_n)=\Int_Y(B_n\cap Y)
\ 
\mbox{ for each }
\ 
n\in \N.
\end{equation}  

\begin{claim}
\label{nested:claim}
The sequence $\{B_n:n\in\N\}$ is nested in $X$.
\end{claim}
\begin{proof}
Let $n\in\N$ be arbitrary.
We want to show that $\Cl _X(B_{n+1})\subseteq \Int_X(B_n)$.
Note that 
\begin{equation}
\label{eq:two:inclusions}
\Cl _X(B_{n+1})\cap Y=\Cl _Y(B_{n+1}\cap Y)\subseteq \Int_Y(B_n\cap Y)
\end{equation}
by Lemma \ref{basic:sets}(i) and \eqref{nested:in:Y}.

Take a point $x=(x_i)_{i\in I}\in \Cl _X(B_{n+1})$. 
The set $J=\mathrm{supp}(B_{n+1})\cup \mathrm{supp}(B_n)$ is finite. Since $\pi_J(Y)=X_J$
by our assumption,
there exists $y=(y_i)_{i\in I}\in Y$ such that 
$y_i=x_i$ for every 
$i\in J$.
Since $x\in \Cl _X(B_{n+1})$ and $\mathrm{supp}(\Cl _X(B_{n+1}))\subseteq \mathrm{supp}(B_{n+1})\subseteq J$
by Remark \ref{remarks:closures:and:interiors}, this gives $y\in  \Cl _X(B_{n+1})$.
Combining this with 
\eqref{eq:two:inclusions},
we get
$y\in \Int_Y(B_n\cap Y)$. Therefore,
$y\in V\cap Y\subseteq B_n\cap Y$ for some basic open neighbourhood $V$ of 
$y$ in $X$.
Now $V\subseteq B_n$ by Lemma \ref{basic:sets}(ii),
so $y\in \Int_X(B_n)$.
Since $\mathrm{supp}(\Int_X(B_n))\subseteq \mathrm{supp}(B_n)\subseteq J$ by Remark \ref{remarks:closures:and:interiors}
and $x_i=y_i$ for all $i\in J$, this gives $x\in \Int_X(B_n)$.
\end{proof}

Since $\{\cal B_n:n\in\N\}$ is a centered sequence in $X$
by Claim \ref{centered:claim}, and $B_n\in\cal B_n$ for all $n\in\N$,
from Claim \ref{nested:claim} we conclude that 
$\bigcap _{n\in\N}B_n\not=\emptyset$.
Fix $x=(x_i)_{i\in I}\in \bigcap _{n\in\N}B_n$.
The set $J=\bigcup_{n\in\N} \mathrm{supp}(B_n)$ is at most countable.
Since $\pi_J(Y)=X_J$ by our assumption, there exists 
$y\in Y$ such that $y_i=x_i$ for all $i\in J$.
Then
$y\in Y\cap \bigcap _{n\in\N}B_n=\bigcap _{n\in\N}(B_n\cap Y)
=\bigcap _{n\in\N}C_n\not=\emptyset$.
This proves that  $\{\cal C_n:n\in\N\}$ is a centered sequence in $Y$.
\endproof

For countably compact families we have the following lemma.

\begin{lemma}
\label{product_sequences2}
Let $X=\prod_{i\in I} X_i $ be a product of non-empty sets $X_i$.
For every $i\in I$, let 
$\cal B_i$ be a \cc\/ family
in $X_i$ such that $X_i\in \cal B_{i}$.
Define
\begin{equation}
\label{rectangular:products:2}
\B=\{B\subseteq X: B=\prod_{i\in I} B_i
\mbox{ is a basic subset of } X
\mbox{ such that }
B_i\in \cal B_i
\mbox{ for all }
i\in I\}.
\end{equation}
Let $Y$ be a subset of $X$ such that
$\pi _J(Y)=X_J$ 
for every at most countable subset $J$ of $I$.
Then 
$\cal C=\{B\cap Y:B\in\cal B\}$
is a \cc\ family in $Y$ such that $Y\in{\cal C}$.
\end{lemma}

\proof
For every $i\in I$, equip $X_i$ with the discrete topology.
For each $i\in I$ and every $n\in\N$ define  
$\cal B_{i,n}=\cal B_i$.
By Remark
\ref{compact:are:strongly:complete}(i), the sequence $\{\cal B_{i,n}:n\in \N\}$ is centered in $X_i$.
By Lemma \ref{product_sequences}, 
the sequence $\{\cal B_{n}:n\in \N\}$ as in the conclusion of this lemma 
is centered in $X$. 
Observe that $\cal B= \cal B_{n}$ for every $n\in \N$ and the elements of $\cal B$ are clopen in $X$. 
Therefore, $\cal B$ is \cc\ in $X$ by Remark
\ref{compact:are:strongly:complete}(ii). 
\endproof

\begin{corollary}\label{product:of:Telgarsky:spaces}\label{productSO}
Let $\cal P$ be one of the following properties: 
\Telgarsky\/, \Op\/, \Tp\/, \SOp\/, 
strongly \Telgarsky\/, strongly \Op\/, strongly \Tp\/, strongly \SOp\/,
Telg\'arsky \cc\/, Oxtoby \cc\/, Todd \cc\/ or S\'anchez-Okunev \cc\/.  
Let $Y$ be a subspace of 
the product $X=\prod_{i\in I} X_i $ of non-empty topological spaces $X_i$
such that 
$\pi_J(Y)=X_J$
for every at most countable subset $J$ of $I$.
If $X_i$ has property $\cal P$ for all $i\in I$, 
then $Y$ has property  $\cal P$ as well.
\end{corollary}

\proof
For 
every 
$i\in I$, 
let 
$\{\cal B_{i,n}:n\in\N\}$ 
be 
a centered sequence in $X_i$ such that each $\B_{i,n}$
for $n\in\N$
is one of the following:
\begin{enumerate}
\item a base of $X_i$;
\item a pseudobase of $X_i$ consisting of open sets;
\item a pseudobase of $X_i$;
\item a pseudobase of $X_i$ consisting of zero-sets.
\end{enumerate}
By Corollary \ref{no:loss:of:generality}, we may assume that $X_i\in \cal B_{i,n}$ whenever $n\in \N$ and $i\in I$.
By Lemma \ref{product_sequences}, 
the sequence
$\{\cal C_{n}:n\in\N\}$ as in the conclusion of that lemma
is centered in $Y$. Furthermore, one can  easily check that each $\cal C_n$ for $n\in\N$ has the correspondent property (1)--(4)
above (with $X_i$ replaced by $Y$).
This argument proves our corollary in case of 
\Telgarsky ness, \Op ness, \Tp ness and \SOp ness, respectively.

Assume, in addition, that all $\B_{i,n}$ ($n\in\N$) are the same, for each $i\in I$.
Then the resulting ${\cal C}_n$ ($n\in\N$) from the conclusion of Lemma\ref{product_sequences} are also the same.
This proves our corollary in case of strong
\Telgarsky ness, strong \Op ness, strong \Tp ness and strong \SOp ness, respectively.

Finally, for every $i\in I$,
let $\cal B_{i}$ be a \cc\/ 
family of subsets of $X_i$
which 
is one of (1)--(4) above.
Without loss of generality, we may assume that $X_i\in \cal B_{i}$
for $i\in I$. 
By Lemma \ref{product_sequences2}, 
 the family $\cal C$ as in the conclusion of that lemma
is a \cc\/ 
family in $Y$. 
Moreover, $\cal C$  has the correspondent property (1)--(4)
above (with $X_i$ replaced by $Y$).
This argument proves our corollary in case of 
Telg\'arsky \ccness, Oxtoby \ccness, Todd \ccness\ and S\'anchez-Okunev \ccness. 
\endproof

\begin{remark}\label{countably_projection:G_delta}
{\em Let $Y$ be a subspace of the product $X=\prod_{i\in I}X_i$ of metric spaces. Then the following conditions are
equivalent:

(i) 
$\pi _J(Y)=X_J$
for every  countable subset $J$ of $I$;

(ii) $Y$ is $G_\delta $-dense in $X$.} 

Indeed, (i) always implies (ii), and to prove that (ii) implies (i) it is enough to observe that all one-point subsets of $X_i$ are
$G_\delta $-sets in $X_i$ for every $i\in I$. 
\end{remark}

\medskip
\noindent
{\bf Proof of Theorem \ref{equivalence:for:dense:subgroups:without:cech}:}
(i)$\Rightarrow $(viii) and (ii)$\Rightarrow $(viii). Since $H$ is c.c.c, $G$ is also c.c.c. 
By \cite[Corollary 8.1.15]{AT}, $G$ is $\R$-factorizable. By Corollary \ref{polish_fact}, $G$ is Polish factorizable.

(viii)$\Rightarrow $(ix). 
It follows from Theorem \ref{O-omega-discrete}(a) and Remark
\ref{countably_projection:G_delta} that 
$G$ is $G_\delta $-dense in $H$.
Moreover, by item (b) of Theorem \ref{necessary:condition:for:completeness:of:a:dense:subgroup},
$H_i$ must be \v{C}ech-complete for every $i\in I$.
 
(ix)$\Rightarrow $(iii), (ix)$\Rightarrow $(vi) and (ix)$\Rightarrow $(vii) 
By Theorem \ref{metric:Todd:are:Cech-complete},
$H_i$ is S\'anchez-Okunev \cc\/, Oxtoby \cc\/ and \Telgarsky\/,
for every $i\in I$. 
By Corollary \ref{product:of:Telgarsky:spaces}, $H$ is S\'anchez-Okunev \cc\/, Oxtoby \cc\/ and \Telgarsky\/. 
Again, by Corollary \ref{productSO}, $G$ is S\'anchez-Okunev \cc\/, Oxtoby \cc\/ and \Telgarsky\/.

Using Diagram 2 we can conclude that all the conditions are equivalent.
\qed
\endproof

\section{Open questions}
\label{sec:questions}

It is not clear whether  the Telg\'arsky series of completeness properties can be added to the list of equivalent conditions
in Theorem \ref{coincidence:for:metric:spaces}:
\begin{question}
Let $X$ be a metric space having any of the equivalent conditions  
from Theorem \ref{coincidence:for:metric:spaces}.
Is $X$ \Telgarsky? Is $X$ \sTep? Is $X$ Telg\'arsky \cc?
\end{question}
The answer to the first question of the three is positive when $X$ is a topological group; see Theorem \ref{metric:Todd:are:Cech-complete}. Nevertheless, the answer to the other two questions remains unclear even in case of metric groups.
\begin{question}
Let $X$ be a metric group having any of the equivalent conditions  
from Theorems \ref{coincidence:for:metric:spaces} and \ref{metric:Todd:are:Cech-complete}.
Is $X$ \sTep? Is $X$ Telg\'arsky \cc?
\end{question}

\begin{question}
(i) Is every \v{C}ech-complete space $X$ \sTep? What if one additionally assumes $X$ to be metric?

(ii) Is every \v{C}ech-complete topological group $X$ \sTep? What if one additionally assumes $X$ to be metric?
\end{question}

The authors were unable to determine if zero-dimensionality 
of the space is essential in Proposition 
\ref{zero:dimensional:psc}.

\begin{question}
\label{que:zero:dimensional:psc}
Is every (locally) pseudocompact space Telg\'arsky \cc\ or Oxtoby \cc?
\end{question}

A positive answer to this question would allow one to drop the assumption of zero-dimensionality from Corollary \ref{zero:dimensional:precompact}.

Precompact groups are $\mathbb{R}$-factorizable and $\mathbb{R}$-factorizable groups are $\omega$-bounded in Guran's sense \cite{Guran}.
This suggests  the following natural question:
\begin{question}
\label{omega-bounded:question}
Do some of the equivalences from Theorem \ref{equivalent:conditions:for:precompact} remain valid in case when $G$ is $\mathbb{R}$-factorizable  (or even when $G$ is $\omega$-bounded)? 
\end{question}

Since the real line $\mathbb{R}$ is $\mathbb{R}$-factorizable \cite{Tkachenko},
Example \ref{R:example} shows that an $\mathbb{R}$-factorizable S\'anchez-Okunev \cc\ (separable metric) group need not be weakly pseudocompact.

\begin{question}
Which of the properties (iii)--(x) from Theorem \ref{equivalent:conditions:for:precompact} remain equivalent in subgroups of locally compact groups?
\end{question}

The following question can be considered as a ``heir'' of Conjecture \ref{conjecture:Tkachenko}.

\begin{question}
Let $G$ be a subgroup of a locally compact group. If $G$ has an open neighbourhood whose closure is weakly pseudocompact, does $G$ also has an open neighbourhood with pseudocompact closure?
\end{question}

A particular special case of Question \ref{omega-bounded:question} is especially interesting:
{\em When is a weakly pseudocompact topological group pseudocompact?\/} Due to Theorem 
\ref{equivalent:conditions:for:precompact}, this question is equivalent to asking when a weakly pseudocompact topological group is precompact.

\begin{question}
Is every weakly pseudocompact $\omega$-bounded group precompact (equivalently, pseudocompact)?
Is every weakly pseudocompact $\mathbb{R}$-factorizable group precompact?
\end{question}

The answer is positive for Lindel\"of groups as weakly pseudocompact Lindel\"of spaces are compact.

Theorem \ref{necessary:condition:for:completeness:of:a:dense:subgroup} and 
Diagram 2 suggest the following 

\begin{question}
If some dense subgroup $G$ of a product $\prod_{i\in I} H_i$ of separable metric groups $H_i$ is \Tp, must all groups $H_i$ be \v{C}ech-complete?
\end{question}

\medskip
\noindent
{\bf Acknowledgement:\/} This paper was written during the first listed author's stay at the Department of Mathematics of Faculty of Science of Ehime University (Matsuyama, Japan)
in the capacity of Visiting Foreign Researcher. 
He would like to thank the host institution for its hospitality.

\end{document}